\documentclass[12pt,twoside]{article}
\usepackage{amssymb,amsmath,mathrsfs,txfonts,graphicx,color}
\usepackage{epstopdf}
\usepackage{cite}
\usepackage{bbm}
\usepackage{bm}
\usepackage{amsfonts}

\usepackage{amsmath}
\usepackage{multirow}
\usepackage{booktabs}

\usepackage{cite}
\usepackage{multirow}
\usepackage{float} \usepackage[colorlinks=true]{hyperref}
\let\oldsection\section
\renewcommand\section{\setcounter{equation}{0}\oldsection}

\newtheorem{theorem}{Theorem}[section]
\newtheorem{lemma}{Lemma}[section]
\newtheorem{proposition}{Proposition}[section]

\newtheorem{remark}{Remark}[section]

\allowdisplaybreaks

\begin{document}
\raggedbottom
\title{\Large\bf Global boundedness of the chemotaxis system with
weakly singular sensitivity and nonlocal term in any dimension}
\author{Yang Cao$^{1}$, Qingchun Li$^{1}$, Jing Zhang$^{2}$\thanks{Corresponding author. {\it E-mail}:  zj188838@163.com}
\footnotemark[0]\\
\footnotesize $^{1}$School of Mathematical Sciences, Dalian University of Technology,
Dalian 116024, P. R. China \\
\footnotesize $^{2}$School of Mathematical Sciences, China West Normal University,
Nanchong 637009, P. R. China}
\date{}

\maketitle

\begin{abstract}
This paper is concerned with the parabolic-elliptic chemotaxis system involving weakly singular sensitivity and a nonlocal source:
$u_t=\Delta u-\chi\nabla\cdot\left(\frac{u}{v^k}\nabla v\right)
+u^\alpha\left(r-\mu\int_\Omega u^\beta\,dx\right)$ and $0=\Delta v-v+u^\gamma$
under homogeneous Neumann boundary conditions in a smooth bounded domain
\(\Omega\subset\mathbb{R}^N\) with \(N\ge1\), where \(\chi,r,\mu,\gamma>0\), \(k\in(0,1)\) and \(\alpha,\beta\ge1\).
In view of the enhanced aggregation induced by singular sensitivity and the suppression of excessive population
growth by the nonlocal damping, we identify sufficient conditions under which solutions remain globally bounded.
We show that classical solutions are globally bounded in the following cases:
\par\smallskip
\begin{center}
\small
\setlength{\tabcolsep}{5pt}
\setlength{\arrayrulewidth}{0.3pt}
\renewcommand{\arraystretch}{1.25}
\begin{tabular}{c|c|c|c}
\multirow{2}{*}{$\beta>1$}
& \multirow{2}{*}{$\gamma=1$}
& $N=1$
& $1\leq\alpha<1+2\beta$
\\
\cline{3-4}
&
& $N\geq2$
&
$\begin{gathered}
1\leq\alpha<2,\quad
\alpha+\beta>2+\tfrac{N}{2},\\[-0.5mm]
\text{or}\quad
2\leq\alpha<1+\tfrac{2\beta}{N}
\end{gathered}$
\\
\hline
\multirow{2}{*}{$\beta=1$}
& $\gamma=1$
& $N=1$
& \multirow{2}{*}{
$\begin{array}{l@{\quad}l@{\quad}l}
\text{Case 1:} & 1\leq\alpha<1+\tfrac{2}{N} & \text{if}\quad m_0<\tfrac{r}{\mu},\\[4pt]
\text{Case 2:} & \alpha\geq1               & \text{if}\quad m_0\geq\tfrac{r}{\mu}
\end{array}$
}
\\
\cline{2-3}
& $0<\gamma<\tfrac{2}{N}$
& $N\geq2$
&
\\
\end{tabular}
\end{center}
\par
\smallskip
\noindent Here \(m_0:=\int_\Omega u_0\) denotes the initial total mass.
It is worth noting that both the uniform-in-time \(L^p\)-estimate and the \(L^\infty\)-bound obtained
by Moser iteration are derived without first establishing a uniform positive lower bound for \(v\).
\end{abstract}

\newpage
\flushbottom


\section{Introduction and main results}
In this paper, we consider the following parabolic-elliptic chemotaxis system
with weakly singular sensitivity and a nonlocal source
\begin{align} \label{1.3}\allowdisplaybreaks
 \left\{
    \begin{array}{llll}
        \displaystyle u_t=\Delta u-\chi\nabla\cdot(\frac{u}{v^k}\nabla v)+u^\alpha(r-\mu\int_{\Omega}u^\beta), &&x \in\Omega, t>0,
        \\
        \displaystyle 0=\Delta v-v+u^\gamma, &&x \in \Omega, t>0,
        \\
        \displaystyle \frac{\partial u}{\partial \nu}=\frac{\partial v}{\partial \nu}=0,
        &&x\in\partial\Omega,
    \end{array}
 \right.
\end{align}
together with the initial function $u_0(x)=u(x,0)$ for all $x\in\bar{\Omega}$ satisfying
\begin{align}\label{1.4}
u_{0} \in C^{0}(\bar {\Omega}), \quad u_{0} \geq 0 , \quad \int_{\Omega} u_{0} > 0,
\end{align}
where \(\Omega\subset\mathbb{R}^N\) with \(N\ge1\) is a bounded domain with smooth boundary.
The parameters satisfy \(\chi,r,\mu,\gamma>0\), \(k\in(0,1)\) and \(\alpha,\beta\ge1\).
Here $u=u(x,t)$ denotes the density of the mobile population and $v=v(x,t)$ is the concentration of the chemical signal.
The coefficient \(\chi\) represents the chemotactic sensitivity, while \(r\) and \(\mu\) characterize population growth and nonlocal damping.
The term \(u^\gamma\) describes the production of the chemical signal.

Chemotactic pattern formation stems from the interplay between signal-directed migration and random diffusion,
which may produce cellular aggregation and spatial structures. Keller and Segel first interpreted
\textit{Dictyostelium discoideum} aggregation as a cell-signal instability \cite{ref1},
and later proposed a model for traveling bacterial bands \cite{ref2,ref3}.
After nondimensionalization, the system takes the form
\begin{align} \label{1.1}\allowdisplaybreaks
 \left\{
    \begin{array}{llll}
        \displaystyle u_t=\Delta u-\nabla\cdot\bigl(uS(v)\nabla v\bigr)+f(u), &&x \in\Omega,
        \\
        \displaystyle \tau v_t=\Delta v-v+u,   &&x \in\Omega,
    \end{array}
 \right.
\end{align}
with homogeneous Neumann boundary conditions.
Here \(S(v)\) denotes the signal-dependent chemotactic sensitivity in the cross-diffusion term
$-\nabla\cdot\bigl(uS(v)\nabla v\bigr)$, which describes chemotactic movement.
The function \(f(u)\) characterizes the population kinetics,
while \(-v\) and \(+u\) represent linear signal degradation and production by the cells, respectively.

The sensitivity \(S(v)\) is often assumed to be constant,
which corresponds to a cellular response to absolute changes in signal concentration.
Responses in many sensory systems instead depend more naturally on relative changes in the stimulus.
The Weber-Fechner law states that perceived intensity is approximately proportional to the logarithm of the stimulus,
which implies that the direction and magnitude of the perceived signal are represented by \(\nabla(\log v)=\frac{\nabla v}{v}\).
This relation motivates the classical singular sensitivity
\begin{align*}
S(v)=\frac{\chi}{v}, \quad \chi>0.
\end{align*}
The singularity in $S(v)$ has motivated extensive studies of the global existence, boundedness
and finite-time blow-up of solutions.
For the chemotaxis model in the parabolic-elliptic setting without a source term,
namely \(\tau=0\) and \(f(u)=0\), Nagai and Senba \cite{ref4} considered radially symmetric solutions when \(\Omega\) is a ball.
They proved that radial solutions to system \eqref{1.1} are globally bounded for arbitrary \(\chi>0\) when \(N=2\)
and for \(\chi<\frac{2}{N-2}\) when \(N\ge3\),
and constructed finite-time radial blow-up solutions when \(\chi>\frac{2N}{N-2}\).
Fujie, Winkler, and Yokota \cite{ref49} then obtained globally bounded classical solutions of system \eqref{1.1}
under \(\chi<\frac{2}{N}\) for all \(N\ge2\).
In two dimensions, Fujie and Senba \cite{ref7} removed this restriction
and established global boundedness for arbitrary \(\chi>0\).
More recently, Kurt \cite{ref10} slightly improved the classical threshold
\(\chi<\frac{2}{N}\) for global boundedness.
For related results on fully parabolic chemotaxis systems with \(\tau=1\), we refer the reader to
\cite{ref11,ref13,ref16}.
In the presence of a logistic source, the damping effect regulates population growth
and thereby suppresses excessive aggregation and prevents blow-up.
For the two-dimensional parabolic-elliptic model with \(f(u)=ru-\mu u^2\), Fujie et al. \cite{ref18}
established the global existence of classical solutions for any \(\chi,\mu>0\)
and uniform boundedness under suitable conditions on \(r\).
Cao et al. \cite{ref19} proved solutions to system \eqref{1.1} converge exponentially to a positive
constant equilibrium under stronger lower bounds on \(r\) and additional assumptions on the initial data.
Kurt and Shen \cite{ref24} extended the analysis to space- and time-dependent logistic coefficients
and obtained global existence and boundedness in arbitrary dimensions under suitable assumptions.
The corresponding fully parabolic problems can be found in \cite{ref20,ref21,ref22,ref23}.

As a generalization of the classical singular sensitivity, the weakly singular
form \(S(v)=\frac{\chi}{v^k}\) with \(0<k<1\) has attracted considerable attention in recent years.
Zhao \cite{ref27} studied the following two-dimensional parabolic-elliptic system
\begin{align} \label{1.2.1}\allowdisplaybreaks
 \left\{
    \begin{array}{llll}
        \displaystyle u_t=\Delta u-\chi\nabla\cdot(\frac{u}{v^k}\nabla v)+ru-\mu u^2, &&x \in\Omega,
        \\
        \displaystyle 0=\Delta v-v+u, &&x \in \Omega.
    \end{array}
 \right.
\end{align}
For every \(k\in(0,1)\), he showed that system \eqref{1.2.1} admits a unique globally bounded classical solution
whenever \(\mu>\mu^*(u_0,\Omega,\chi,k,r)\).
His extension argument uses a time-decaying lower bound for \(v\),
whereas uniform boundedness is established without a time-uniform positive lower bound.
Kurt \cite{ref29} then extended the result of \cite{ref27} to arbitrary dimensions
and introduced a threshold \(\mu^*\) independent of the initial data.
When \(\mu>\mu^*\), global classical solutions exist for every \(N\ge1\) and are uniformly bounded for \(N\ge2\) if \(k<\frac12+\frac1N\).
Unlike \cite{ref27}, the \(L^p\) estimates in \cite{ref29} require no preliminary pointwise lower bound for \(v\).
In two dimensions, Le \cite{ref30} replaced $-\mu u^2$ with logarithmically weakened sub-logistic damping
and removed the largeness condition on \(\mu\).
Le and Kurt \cite{ref31} later eliminated the restriction linking \(k\) to the dimension.
They established global boundedness of solutions to system \eqref{1.2.1} for every \(N\ge3\)
and \(k\in(0,1)\) provided that $\mu$ exceeds an explicit threshold.
Le \cite{ref32} later extended the entire range \(k\in(0,1)\) to the fully parabolic system
in dimensions \(N\ge3\) and established global boundedness for sufficiently large \(\mu\).
Both studies directly controlled the singular terms without requiring a uniform-in-time positive lower bound for \(v\).

The local logistic term makes population growth depend only on local density,
whereas ecological dynamics may also be shaped by total resources, population size and long-range interactions.
Such global feedback can be represented at the mean-field level
by the nonlocal source \(u^\alpha(r-\mu\int_\Omega u^\beta)\).
When \(\beta>1\), the term \(\int_\Omega u^\beta\) in the nonlocal source can
control higher powers of the cell density and suppress excessive aggregation.
Bian et al. \cite{ref34} investigated the following parabolic-elliptic system
\begin{align} \label{1.2}\allowdisplaybreaks
 \left\{
    \begin{array}{llll}
        \displaystyle u_t=\Delta u-\nabla\!\cdot(u\nabla v)+u^\alpha\left(1-\int_\Omega u^\beta\right), &&x \in\Omega,
        \\
        \displaystyle 0=\Delta v-v+u, &&x \in \Omega,
    \end{array}
 \right.
\end{align}
where \(\alpha\ge1\).
For \(N\ge3\) and \(\beta>1\), they proved that the system \eqref{1.2} admits a unique globally bounded classical solution
if either \(2\le\alpha<1+\frac{2\beta}{N}\) or \(\alpha<2\) and \(\frac{N+2}{N}(2-\alpha)<1+\frac{2\beta}{N}-\alpha\).
The parabolic-elliptic results were then extended to the fully parabolic setting by Chiyo et al. \cite{ref36}.
They established global boundedness of classical solutions for every \(N\ge1\) if either \(1\le\alpha<2\)
and \(\beta>\frac{N+4}{2}-\alpha\) or \(\beta>\frac N2\) and \(2\le\alpha<1+\frac{2\beta}{N}\).
Further results on finite-time blow-up and multispecies nonlocal regulation
are presented in \cite{ref37,ref42,ref43}.
With the singular sensitivity \(\frac{1}{v}\), Du et al. \cite{ref39}
studied a fully parabolic chemotaxis system involving a nonlocal source.
For \(N\ge2\), they established the existence and uniform boundedness of
a unique global classical solution with suitable parameter conditions.

Biologically, singular chemotactic sensitivity and nonlocal population regulation
influence population dynamics through distinct mechanisms.
On the one hand, chemotaxis models with singular sensitivity can describe traveling bands of chemotactic bacteria,
front propagation and spatially heterogeneous structures \cite{ref20,ref18,ref11-1}.
In the presence of signal consumption or a local logistic source,
such models may also exhibit single or multiple spikes and may undergo chemotactic collapse under suitable conditions \cite{ref4,ref11-2}.
On the other hand, the nonlocal source regulates birth and death through shared-resource consumption and total population size \cite{ref36,ref11-3}.
Appropriate nonlocal damping can prevent cell aggregation and chemotactic collapse,
and may even drive the population toward a spatially homogeneous state \cite{ref11-3}.
If this damping is insufficient, finite-time blow-up may still occur \cite{ref10}.
Moreover, intraspecific competition for limited resources over a broad spatial range
may generate steady spatially periodic structures, periodic standing waves and periodic travelling waves \cite{ref11-4}.
During population invasion, a travelling wavefront may develop a local high-density hump and leave a nonuniform steady state in its wake \cite{ref11-5}.
These biological phenomena are closely related to the distinct mathematical roles of the two mechanisms.
Although the singularity of \(v^{-k}\) for \(0<k<1\) is milder near \(v=0\) than that of \(v^{-1}\),
the sensitivity remains unbounded as \(v\to0\) and may therefore significantly enhance chemotactic aggregation
in regions of low signal concentration.
For the nonlocal source \(u^\alpha\big(r-\mu\int_\Omega u^\beta\big)\) with \(\beta>1\),
the term \(\int_\Omega u^\beta\) provides higher-order nonlocal damping,
which can prevent blow-up under suitable conditions \cite{ref36,ref11-3}.
When \(\beta=1\), the feedback depends only on the total population,
and its role depends on the initial mass and the exponent \(\alpha\).
For \(\int_\Omega u_0\ge r/\mu\), the source remains nonpositive and acts as damping.
When \(\int_\Omega u_0<r/\mu\), it initially promotes growth and aggregation may intensify as \(\alpha\) increases.
The interplay between singular sensitivity and the nonlocal source can lead to
more intricate dynamical behaviors and increased complexity of the mathematical analysis.
To the best of our knowledge, relatively few results are available for chemotaxis systems incorporating both mechanisms,
and a recent result can be found in \cite{ref39}.
Accordingly, this paper aims to establish sufficient conditions for global existence and uniform boundedness of solutions
to the chemotaxis system involving both mechanisms.
These results will lay a foundation for further studies of pattern formation, front propagation and long-time behavior.

Before stating the theorem, we introduce the parameter condition \(\mathcal{H}\), which is satisfied if one of the following holds:
\begin{itemize}
\item[\(\mathrm{(H_1)}\)]
Suppose that \(\beta>1\) and \(\gamma=1\). If \(N=1\), assume that
\[
1\leq\alpha<2
\quad\text{or}\quad
2\leq\alpha<1+2\beta.
\]
If $N\geq2$, assume that
\[
1\leq\alpha<2,\quad\alpha+\beta>2+\frac N2,\quad\text{or}\quad
2\leq\alpha<1+\frac{2\beta}{N}.
\]

\item[\(\mathrm{(H_2)}\)]
Suppose that \(\beta=1\), with \(\gamma=1\) when \(N=1\) and \(0<\gamma<\frac{2}{N}\) when \(N\ge2\).

If $\int_\Omega u_0<\frac{r}{\mu}$, we further assume that $1\leq\alpha<1+\frac{2}{N}$.

If $\int_\Omega u_0\geq\frac{r}{\mu}$, we only require \(\alpha\geq1\).
\end{itemize}

Our first main result provides uniform-in-time $L^p(\Omega)$-bounds for $u$.
The proof distinguishes between the cases $\beta>1$ and $\beta=1$:
the higher-order terms are controlled by the nonlocal dissipation in the former case and by the evolution of the total mass in the latter.
Combining these controls with energy and interpolation estimates yields the following result.
\begin{theorem}\label{th3.2}
Let the initial datum \(u_0\) satisfy \eqref{1.4}. Suppose that
\(\chi>0\), \(k\in(0,1)\) and \(\mathcal{H}\) holds. Then for
every \(p\geq2\), there exists a constant \(C_p>0\) such that
\[
\sup_{0<t<T_{\max}}\int_\Omega u^p(x,t)\leq C_p,
\]
where \(C_p\) is independent of \(t\).
\end{theorem}

We next use the $L^p$-bounds from Theorem \ref{th3.2}, together with elliptic regularity and a Moser iteration,
to obtain a uniform $L^\infty$-bound for $u$ without first establishing a positive lower bound for $v$.
This bound then yields time-uniform positive lower bounds for $\int_\Omega u$ and $v$.
Consequently, the extensibility criterion gives the following global existence and uniform boundedness result.
\begin{theorem}\label{th3.4}
Under the assumptions of Theorem \ref{th3.2}, the solution of
\eqref{1.3} exists globally and remains uniformly bounded in the sense that
there exists \(C_\infty>0\) such that
\[
\|u(\cdot,t)\|_{L^\infty(\Omega)}\leq C_\infty,
\]
for all \(t>0\).
\end{theorem}

\begin{remark}
Since the sensitivity $v^{-k}$ becomes singular as $v\to0$, a uniform-in-time positive lower bound
for \(v\) is key to controlling the singular chemotactic cross-diffusion.
For parabolic-elliptic chemotaxis systems with a local logistic source, a typical approach starts by deriving an \(L^p\)-bound for \(u\).
One then tests the cell equation with \(u^{-1}\), thereby obtaining a positive lower bound for \(\int_\Omega u\) from a lower estimate for \(\frac{\mathrm d}{\mathrm dt}\int_\Omega\ln u\).
The elliptic equation subsequently provides a positive lower bound for \(v\), which can be used to estimate \(\|u\|_{L^\infty(\Omega)}\).
However, the presence of the nonlocal source makes the logarithmic test
produce a coupling term involving \(\int_\Omega u^\beta\) and \(\int_\Omega u^{\alpha-1}\), which cannot be controlled solely through the total mass.
Thus, the above argument does not directly provide the required
uniform-in-time lower bound and therefore fails to close the \(L^\infty\)-estimate.
We consequently seek a direct \(L^\infty\)-estimate through Moser iteration
without first using a uniform-in-time positive lower bound for \(v\).
Although Le and Kurt also employed Moser iteration \cite{ref31}, their method
relies on the power structure generated by the local quadratic logistic source.
However, the nonlocal source in system \eqref{1.3} produces the higher-order
growth term \(\int_\Omega u^{p+\alpha-1}\) and couples it with \(\int_\Omega u^\beta\) in the damping term.
When \(\alpha>2\), estimating the growth term directly imposes an additional
restriction on \(\alpha\), whereas discarding the coupled damping term loses the
dissipation needed to derive the desired bound.
To this end, we choose the exponent at each step according to the two powers
\(\beta\) and \(p+\alpha-1\) in the nonlocal dissipation term.
This allows us to complete the Moser iteration and obtain an \(L^\infty\)-bound for \(u\).
\end{remark}

\begin{remark}
For $\beta>1$ and $\gamma=1$, we establish the global boundedness
of the solution to system \eqref{1.3} which involves weakly singular sensitivity and a nonlocal source for \(N\geq1\).
Compared with the constant sensitivity in the nonlocal chemotaxis model studied by Bian et al. \cite{ref34},
the weakly singular sensitivity
\(v^{-k}\) considered here further enhances chemotactic aggregation at low signal
concentrations. Under this stronger aggregation mechanism, our result for
\(N\geq3\) still coincides with that of \cite{ref34}.
The approach adopted by Bian et al. relies on the finite Sobolev critical
exponent \(\frac{2N}{N-2}\), and hence their interpolation framework is
applicable only when \(N\geq 3\).
To address the lower-dimensional cases, we employ dimension-dependent Gagliardo-Nirenberg inequalities and choose intermediate and iteration exponents adapted to the structure of the nonlocal dissipation.
These choices enable us to complete the \(L^p\)-estimates and Moser iteration, thereby establishing global boundedness for \(N=1,2\).
\end{remark}

\begin{remark}
In the case \(\beta=1\), we prove global boundedness of solutions to a parabolic-elliptic chemotaxis system with weakly singular sensitivity and a nonlocal Fisher--KPP reaction term. For \(\gamma=1\), \(N=1\) and \(\int_\Omega u_0<\frac{r}{\mu}\), compared with the result obtained by Bian et al. \cite{ref46} for the nonlocal Fisher--KPP equation without chemotaxis, we still find a wider range of parameters for which global boundedness holds in the presence of stronger chemotactic aggregation.
The argument developed by Bian et al. controls the positive source term through interpolation with an intermediate norm, while the associated absorption condition requires \(1\leq\alpha<2\) in one dimension.
In the present paper, we employ weighted estimates derived from the elliptic signal equation together with the total mass estimate and the Gagliardo-Nirenberg inequality, which enables us to absorb the chemotactic and positive source terms into the gradient dissipation.
This refined estimate ultimately establishes global boundedness in one dimension throughout the range \(1\leq\alpha<3\).
\end{remark}

\begin{remark}
For \(\beta=1\), if \(\int_\Omega u_0\ge \frac{r}{\mu}\), the mass dynamics ensures \(r-\mu\int_\Omega u\le0\),
and the nonlocal source therefore acts purely as a damping mechanism.
If \(\int_\Omega u_0<\frac{r}{\mu}\), the source is initially proliferative,
and the condition \(\alpha<1+\frac{2}{N}\) prevents this growth from reinforcing excessive aggregation.
Moreover, higher-dimensional settings are more conducive to aggregation and even to blow-up.
To ensure global existence and boundedness of solutions,
it is necessary that the signal production be reduced from the linear form \(u\)
to the sublinear form \(u^\gamma\) with \(0<\gamma<\frac{2}{N}\),
thereby weakening chemotactic aggregation.
\end{remark}

The rest of this paper is organized as follows.
Section 2 presents the preliminary lemmas needed for the subsequent analysis.
In Section 3, we establish uniform-in-time \(L^p\) estimates for the solution.
Section 4 is devoted to the \(L^\infty\) bound,
which completes the proof of global existence and uniform boundedness.

\section{Preliminaries}
This section is devoted to establishing two basic estimates for \eqref{1.3} and collecting the analytical tools to be used later.
For \(N\geq3\), set
\begin{equation}\label{2.0}
q:=\frac{2N}{N-2}.
\end{equation}
Without loss of generality, we suppose $|\Omega|=1$.
We first establish the local existence and uniqueness of classical solutions to problem \eqref{1.3}.

\begin{proposition} \label{pro4.1}
(Local existence)
Assume that \(\alpha,\beta\ge1\), \(0<\gamma\le1\) and that the initial datum \(u_0\) satisfies \eqref{1.4}.
Then there exist a maximal existence time $T_{\max}\in(0,\infty]$ and
a nonnegative pair $(u(x,t),v(x,t))$ of functions
\[
u\in C\big(\bar\Omega\times(0,T_{\max})\big)
\cap C^{2,1}\big(\bar\Omega\times(0,T_{\max})\big),
\quad
v\in C^{2,0}\big(\bar\Omega\times(0,T_{\max})\big),
\]
satisfying \eqref{1.3} in the classical sense, which is unique unless
$0<\gamma<1$.
Moreover, if $T_{\max}<\infty$, then
\[
\text{either}\quad
\limsup_{t\nearrow T_{\max}}
\|u(\cdot, t)\|_{C^0(\bar\Omega)}
=\infty,
\quad\text{or}\quad
\liminf_{t\nearrow T_{\max}}
\inf_{x\in\Omega}v(x,t)=0.
\]
\end{proposition}

For $\gamma=1$, the proof is based on the Banach fixed-point argument and parabolic regularity,
as detailed in \cite[Proposition 3.1]{ref49} and \cite[Proposition 4]{ref34}.
For \(0<\gamma<1\), the mapping \(s\mapsto s^\gamma\) fails to be locally Lipschitz continuous at \(s=0\),
so the Banach fixed-point argument is not applicable and uniqueness is not guaranteed.
We instead employ a Schauder fixed-point theorem,
following the constructions in \cite[Lemma 3.1]{ref50}, \cite[Lemma 2.1]{ref52} and \cite[Lemma 5.1]{ref51}.
The lower estimate for the Neumann elliptic problem ensures that $v$ remains positive
on each local existence interval, so that the singular part $v^{-k}$ is well defined.
It follows from elliptic and parabolic regularity that the local existence and the extensibility alternative hold.
The details are omitted for brevity.

Next, we derive a uniform estimate for the total mass.
\begin{lemma}\label{le2.1}
Let $r,\mu>0$ and $\alpha\geq1$, and let the initial datum $u_{0}$ satisfy \eqref{1.4}.
If $\beta>1$, then
\[
\int_ {\Omega} u \leq \max\left\{\int_{\Omega}u_0, \left(\frac{r}{\mu}\right)^{\frac{1}{\beta}}\right\};
\]
if $\beta=1$, then
\[
\min\left\{\frac{r}{\mu}, \int_{\Omega} u_0\right\}
\le \int_{\Omega} u
\le \max\left\{\frac{r}{\mu}, \int_{\Omega} u_0\right\}
\]
for all $t \in (0, T _ {\max })$.
\end{lemma}
{\bf Proof.} We first consider the case \(\beta>1\). Integrating the first equation in \eqref{1.3} over \(\Omega\)
and using H\"older's inequality, we obtain
\[
\frac {d}{d t} \int_ {\Omega} u = \int_ {\Omega}u^\alpha\left(r-\mu\int_{\Omega}u^\beta\right)
\leq \int_ {\Omega} u ^ {\alpha} \left[ r -  \mu \left(\int_ {\Omega} u\right) ^ {\beta} \right]
\]
for all $t \in (0, T _ {\max })$.
A comparison argument for scalar ODEs yields the upper bound.
For $\beta=1$, the conclusion follows directly from \cite[Lemma 6]{ref46}.
$\hfill\Box$

\begin{lemma}\label{le2.2}
Let $\gamma>0$. For every $p\geq2$, it holds that
\[
\frac{1}{2}\int_\Omega
u^p\frac{|\nabla v|^2}{v^2}
+\int_\Omega\frac{u^{p+\gamma}}{v}
\leq
2\int_\Omega|\nabla u^{\frac p2}|^2
+\int_\Omega u^p
\]
for all $t \in (0, T _ {\max })$.
\end{lemma}

{\bf Proof.}
Testing the second equation in \eqref{1.3} with
$\frac{u^p}{v}$ and integrating by parts over $\Omega$, we obtain
\begin{align*}
0
&=
\int_\Omega\frac{u^p}{v}
\left(\Delta v-v+u^\gamma\right)\\
&=
-p\int_\Omega
\frac{u^{p-1}}{v}\nabla u\cdot\nabla v
+\int_\Omega
u^p\frac{|\nabla v|^2}{v^2}
-\int_\Omega u^p
+\int_\Omega\frac{u^{p+\gamma}}{v}.
\end{align*}
Hence, by Young's inequality, we have
\begin{align*}
\int_\Omega u^p\frac{|\nabla v|^2}{v^2} +\int_\Omega\frac{u^{p+\gamma}}{v}
&=p\int_\Omega\frac{u^{p-1}}{v}\nabla u\cdot\nabla v+\int_\Omega u^p\\
&\leq 2\int_\Omega|\nabla u^{\frac p2}|^2+\frac12\int_\Omega u^p\frac{|\nabla v|^2}{v^2}
+\int_\Omega u^p.
\end{align*}
The conclusion is therefore obtained.
\(\hfill\Box\)

\begin{lemma}\label{le2.5}
For \(N\geq3\), let \(q\) be defined by \eqref{2.0} and choose
\(1\leq l<s<q\) such that $\frac{s}{l}<\frac{2}{l}+1-\frac{2}{q}$.
For any \(C_{N_1},C_{N_2}>0\) and
\(\omega\in H^1(\Omega)\cap L^l(\Omega)\), one has
$$
\|\omega\|_{L^s(\Omega)}^s\leq C(N)\left( C_{N_1}^{-\frac{\eta s}{2-\eta s}} + C_{N_2}^{-\frac{\eta s}{2-\eta s}}\right)
\|\omega\|_{L^l(\Omega)}^\theta + C_{N_1} \|\nabla \omega\|_{L^2(\Omega)}^2 + C_{N_2} \|\omega\|_{L^2(\Omega)}^2,
$$
where \(C(N)>0\) depends only on \(N\), and
\[
\eta=
\frac{\frac1l-\frac1s}{\frac1l+\frac1N-\frac12}\in(0,1),
\qquad\theta=\frac{2(1-\eta)s}{2-\eta s}.
\]
For \(N=1\) and \(N=2\), the conclusion holds if \(1\leq l<s\) and \(\frac{s}{l}<\frac{2}{l}+2\),
and if \(1\leq l<s\) and \(\frac{s}{l}<\frac{2}{l}+1\), respectively.
\end{lemma}
{\bf Proof.}
The estimate follows from \cite[Lemma 3.1]{ref36}. For \(N=1\),
the above formula gives $\eta=\frac{\frac1l-\frac1s}{\frac1l+\frac12}$,
whereas for \(N=2\), it gives $\eta=1-\frac{l}{s}$.
In both cases, \(\eta s<2\) follows from the stated condition.
\(\hfill\Box\)

\begin{lemma}(\cite{ref47})\label{le2.6}
Let $1 \le r_1 \le r_0 \le r_2 \le \infty$ satisfy
$\frac{1}{r_0}=\frac{\sigma}{r_1}+\frac{1-\sigma}{r_2}$.
If \(u\in L^{r_1}(\Omega)\cap L^{r_2}(\Omega)\), then
\(u\in L^{r_0}(\Omega)\) and
$$
\|u\|_{L^{r_0}(\Omega)} \le \|u\|_{L^{r_1}(\Omega)}^{\sigma} \|u\|_{L^{r_2}(\Omega)}^{1-\sigma}.
$$
\end{lemma}

\begin{lemma}(Lemma 2.3 in \cite{ref48})\label{le2.7}
Let $r\geq1$, $0<s_0\leq m_0\leq\infty$ and $z>0$ satisfy $\frac{1}{r} \le \frac{1}{N} + \frac{1}{m_0}$.
There exists $C_{GN}>0$ such that every $u\in W^{1,r}(\Omega)\cap L^{s_0}(\Omega)$ satisfies
$$
\|u\|_{L^{m_0}(\Omega)} \le C_{GN}\left( \|\nabla u\|_{L^r(\Omega)}^a \|u\|_{L^{s_0}(\Omega)}^{1-a} + \|u\|_{L^z(\Omega)} \right),
$$
where $a = \frac{\frac{1}{s_0} - \frac{1}{m_0}}{\frac{1}{s_0} + \frac{1}{N} - \frac{1}{r}}$.
\end{lemma}

\begin{lemma}(Lemma 7 in \cite{ref46})\label{le2.8}
Let $a_0>1$, $A,B>0$, and let $y\in C^1((0,\infty))$ be nonnegative and satisfy
$$
y'(t) \le A - B y(t)^{a_0}.
$$
For every $t>0$, one has
$$
y(t) \le \left( \frac{A}{B} \right)^{1/{a_0}} + \left[ \frac{1}{B({a_0}-1)t} \right]^{\frac{1}{{a_0}-1}}.
$$
Furthermore, if $ y(0) $ is bounded, then
$$
y(t) \le \max\left( y(0), \left( \frac{A}{B} \right)^{1/{a_0}} \right)
$$
holds for all $t>0$.
\end{lemma}

\begin{lemma}(Lemma 4.1 in \cite{ref45})\label{le2.9}
Let $y_j\in C^1((0,\infty))$, $j=0,1,2,\dots$, be nonnegative functions satisfying
$$
y_j'(t) \le -y_j + m_j \left( y_{j-1}^{\vartheta_1}(t) + y_{j-1}^{\vartheta_2}(t) \right),
$$
where $m_j=\tilde{m}h^{bj}>1$, $\tilde{m},b,h$ are positive and bounded constants, and $0\leq\vartheta_2<\vartheta_1\leq h$.
Suppose further that there exists a bounded constant $\tilde{M}\geq1$ such that $y_j(0)\leq\tilde{M}^{h^j}$ for all $j\geq0$. Then
$$
y_j(t) \le (2\tilde{m})^{\frac{h^j - 1}{h - 1}} h^{b \left( \frac{h(h^j - 1)}{(h - 1)^2} - \frac{j}{h - 1} \right)}
\max\left\{ \sup_{t \ge 0} y_0^{h^j}(t), \tilde{M}^{h^j} \right\}.
$$
\end{lemma}
{\bf Proof.} The proof of \cite[Lemma 4.1]{ref45} remains valid for \(\vartheta_2=0\).

\section{$L^p$-Boundedness}\label{S3}
In this section, we prove the uniform $L^p(\Omega)$-bounds stated in Theorem \ref{th3.2}.
We first derive several basic estimates for \(v\).

\begin{lemma}\label{le2.3}
Let $\gamma=1$, $\rho\geq 3$ and \(2< \kappa < 2 \rho-2\). Then we have
\[
\int_ {\Omega} \frac {| \nabla v | ^ {2 \rho}}{v ^ {\kappa}} \leq \tilde{C} \int_ {\Omega}
\frac {u ^ {\rho}}{v ^ {\kappa - \rho}} + C \int_ {\Omega} v ^ {2 \rho - \kappa}
\]
for all \(t\in(0,T_{\max})\), where $\tilde{C}$ is a positive constant and
\[
\tilde{C} := \left(\frac {4  (\rho - 1) ^ {2}}{2 \rho - \kappa - 2}\right) ^ {\rho}
\cdot \left(\frac {2 (\kappa - 1)}{\kappa - 2}\right) ^ {\frac {\rho}{2}}.
\]
\end{lemma}
{\bf Proof.} The result is proved in \cite[Proposition 3.1]{ref29}.
\(\hfill\Box\)

\begin{lemma}\label{le2.4}
When \(\gamma=1\), the estimate
\[
\int_\Omega v^{p+1}\leq\int_\Omega u^{p+1}
\]
holds for all \(p\geq1\) and \(t\in(0,T_{\max})\).
\end{lemma}
{\bf Proof.} Multiplying the second equation in \eqref{1.3} by $v^{p}$, integrating by parts over $\Omega$,
and applying H\"older's inequality, we obtain
\begin{align*}
    p\int_\Omega v^{p-1} |\nabla v|^2 + \int_\Omega v^{p+1} = \int_\Omega u v^p
    \leq \left(\int_\Omega u^{p+1}\right)^{\frac{1}{p+1}} \left(\int_\Omega v^{p+1}\right)^{\frac{p}{p+1}}.
\end{align*}
The conclusion follows from the nonnegativity of $\int_\Omega v^{p-1} |\nabla v|^2$.
$\hfill\Box$

\begin{lemma}\label{le2.10}
Let \(N\geq2\) and $\gamma \in (0, \frac{2}{N})$. There exists $L_1 > 0$ such that
$$
\|v(\cdot,t)\|_{L^\infty(\Omega)} \le L_1
$$
for all $t \in (0, T_{\max})$.
\end{lemma}
{\bf Proof.}
Choose $s_1$ such that $\frac{N}{2}<s_1<\frac{1}{\gamma}$. Since $\gamma s_1<1$,
applying H\"{o}lder's inequality and Lemma \ref{le2.1}, we have
$$
\|u^\gamma\|_{L^{s_1}(\Omega)}^{s_1}= \int_{\Omega} u^{\gamma s_1}
\le |\Omega|^{1-\gamma s_1} \left( \int_{\Omega} u \right)^{\gamma s_1}\le C,
$$
where \(C>0\) is independent of \(t\).
By elliptic regularity, we further deduce
$$
\|v\|_{W^{2,s_1}(\Omega)}\le C_1 \|u^\gamma\|_{L^{s_1}(\Omega)}\le C_2
$$
with $C_1,C_2>0$ both independent of $t$.
The condition $s_1>\frac{N}{2}$ ensures the continuous embedding $W^{2,s_1}(\Omega) \hookrightarrow L^\infty(\Omega)$.
It follows that
$$
\|v\|_{L^\infty(\Omega)} \le L_1.
$$
$\hfill\Box$

\begin{lemma}\label{le2.11}
Let $N\geq2$ and $0<\gamma<\frac{2}{N}$.
For any $p\geq2$ and $\tilde{\varepsilon}>0$, there exists a constant $C(\tilde{\varepsilon},p)>0$ such that
\[
\int_{\Omega}u^p\frac{|\nabla v|^2}{v}
\le \tilde{\varepsilon}\int_\Omega |\nabla u^{\frac{p}{2}}|^2
+ \frac{2L_1}{e}\int_{\Omega}u^p
+ C(\tilde{\varepsilon},p).
\]
\end{lemma}
{\bf Proof.}
An application of H\"{o}lder's inequality yields
\begin{align}\label{0.1}
    \int_{\Omega}u^p|\nabla v|^2
    \le \left(\int_{\Omega}u^{\frac{p(N+2)}{N}}\right)^{\frac{N}{N+2}}
    \left(\int_{\Omega}|\nabla v|^{N+2}\right)^{\frac{2}{N+2}}.
\end{align}
By Lemma \ref{le2.7} and Lemma \ref{le2.1}, we obtain
\begin{align}\label{0.2}
    \left(\int_{\Omega}u^{\frac{p(N+2)}{N}}\right)^{\frac{N}{N+2}}
    &= \|u^{\frac{p}{2}}\|_{L^{\frac{2(N+2)}{N}}(\Omega)}^2 \nonumber\\
    &\le \left[ C_{GN}\left(\|\nabla u^{\frac{p}{2}}\|_{L^2(\Omega)}^{a}
    \|u^{\frac{p}{2}}\|_{L^{\frac{2}{p}}(\Omega)}^{1-a}
    + \|u^{\frac{p}{2}}\|_{L^{\frac{2}{p}}(\Omega)}\right) \right]^2 \nonumber\\
    &\le C_3\left(\int_{\Omega}|\nabla u^{\frac{p}{2}}|^2\right)^{a} + C_4,
\end{align}
where $a= \frac{\frac{p}{2} - \frac{N}{2(N+2)}}{\frac{p}{2} + \frac{1}{N} - \frac{1}{2}} \in (0,1)$
and \(C_3,C_4>0\) depend on \(p\) and \(N\).
The Sobolev embedding theorem together with the regularity theory for elliptic equations leads to
\begin{align}\label{1}
    \left(\int_{\Omega} |\nabla v|^{N+2}\right)^{\frac{2}{N+2}}= \|\nabla v\|_{L^{N+2}(\Omega)}^2
    \le C_5 \|v\|_{W^{2,q_0}(\Omega)}^2\le C_6 \|u^\gamma\|_{L^{q_0}(\Omega)}^2
    = C_6\left(\int_{\Omega} u^{\gamma q_0}\right)^{\frac{2}{q_0}},
\end{align}
where $q_0 = \frac{N(N+2)}{2N+2}$ and the constants $C_5,C_6>0$ depend only on $N,\Omega$ and are independent of $p$.
If $\gamma q_0 \le 1$, Lemma \ref{le2.1} implies that
$$
\left(\int_{\Omega} |\nabla v|^{N+2}\right)^{\frac{2}{N+2}} \le C
$$
for some $C>0$ independent of $p$ and $t$.
Combining \eqref{0.1} with \eqref{0.2}, we infer that when $\gamma q_0 \le 1$,
there exists $\tilde{a} \in (0,1)$ such that
\begin{align}\label{1.1.0}
    \int_{\Omega} u^p |\nabla v|^2
    \le C_7\left(\int_{\Omega} |\nabla u^{\frac{p}{2}}|^2\right)^{\tilde{a}} + C_8,
\end{align}
where \(C_7,C_8>0\) depend on \(N,p\) and are independent of \(t\).
If $\gamma q_0 > 1$, in view of Lemma \ref{le2.6} and $\gamma < \frac{2}{N}$, we get
$$
\|u\|_{L^{\gamma q_0}(\Omega)}
\le \|u\|_{L^{\frac{p(N+2)}{N}}(\Omega)}^{\sigma} \|u\|_{L^1(\Omega)}^{1-\sigma},
$$
where $\sigma = \frac{1-\frac{1}{\gamma q_0}}{1-\frac{N}{p(N+2)}} \in (0,1)$.
Lemma \ref{le2.1} further implies that there is \(C_9=C_9(N,p)>0\) such that
\begin{align}\label{2}
    \left(\int_{\Omega} u^{\gamma q_0}\right)^{\frac{2}{q_0}}
    = \|u\|_{L^{\gamma q_0}(\Omega)}^{2\gamma}
    \le C_9 \|u\|_{L^{\frac{p(N+2)}{N}}(\Omega)}^{2\gamma\sigma}
    = C_9\|u^{\frac{p}{2}}\|_{L^{\frac{2(N+2)}{N}}(\Omega)}^{\frac{4\gamma\sigma}{p}}.
\end{align}
Combining \eqref{1} and \eqref{2}, we have
\begin{align}\label{3}
    \left(\int_{\Omega} u^{\frac{p(N+2)}{N}}\right)^{\frac{N}{N+2}}
    \left(\int_{\Omega} |\nabla v|^{N+2}\right)^{\frac{2}{N+2}}
    &\le C_{10} \|u^{\frac{p}{2}}\|_{L^{\frac{2(N+2)}{N}}(\Omega)}^2
    \|u^{\frac{p}{2}}\|_{L^{\frac{2(N+2)}{N}}(\Omega)}^{\frac{4\gamma\sigma}{p}} \nonumber\\
    &= C_{10} \|u^{\frac{p}{2}}\|_{L^{\frac{2(N+2)}{N}}(\Omega)}^{2+\frac{4\gamma\sigma}{p}}
\end{align}
with \(C_{10}=C_{10}(N,p)>0\).
Applying Lemma \ref{le2.7} and Lemma \ref{le2.1}, we obtain
\begin{align}\label{4}
   \|u^{\frac{p}{2}}\|_{L^{\frac{2(N+2)}{N}}(\Omega)}^{2+\frac{4\gamma\sigma}{p}}
    &\le C_{GN} \left(\|\nabla u^{\frac{p}{2}}\|_{L^2(\Omega)}^{a}
    \|u^{\frac{p}{2}}\|_{L^{\frac{2}{p}}(\Omega)}^{1-a}
    + \|u^{\frac{p}{2}}\|_{L^{\frac{2}{p}}(\Omega)} \right)^{2+\frac{4\gamma\sigma}{p}} \nonumber\\
    &\le C_{11} \left( \int_{\Omega} |\nabla u^{\frac{p}{2}}|^2 \right)^{\left(1+\frac{2\gamma\sigma}{p}\right)a} + C_{12},
\end{align}
where $a = \frac{\frac{p}{2} - \frac{N}{2(N+2)}}{\frac{p}{2} + \frac{1}{N} - \frac{1}{2}}$
and \(C_{11},C_{12}>0\) depend on \(p\) and \(N\).
The condition $\gamma < \frac{2}{N}$ ensures that $\left(1+\frac{2\gamma\sigma}{p}\right)a < 1$.
For $\gamma q_0 > 1$, combining \eqref{3}, \eqref{4} and \eqref{0.1},
we find that there exists $\hat{a} \in (0,1)$ such that
\begin{align}\label{5}
    \int_{\Omega} u^p |\nabla v|^2
    \le C_{13} \left( \int_{\Omega} |\nabla u^{\frac{p}{2}}|^2 \right)^{\hat{a}} + C_{14},
\end{align}
where \(C_{13},C_{14}>0\) depend on \(N\) and \(p\).
Combining \eqref{1.1.0} and \eqref{5} with Young's inequality, we deduce that for each $\gamma \in (0,\frac{2}{N})$
and $\varepsilon_0>0$, there exists $C(\varepsilon_0, p) > 0$ such that
\begin{align}\label{5.0}
   \int_{\Omega} u^p |\nabla v|^2
   \le \varepsilon_0 \int_{\Omega} |\nabla u^{\frac{p}{2}}|^2 + C(\varepsilon_0, p).
\end{align}
Testing the equation $\Delta v - v + u^{\gamma} = 0$ by $u^p \ln \frac{v}{L_1}$ and integrating over $\Omega$, we obtain
\begin{align*}
   \int_{\Omega} u^p \frac{|\nabla v|^2}{v}
   = -p \int_{\Omega} u^{p-1} \nabla u \cdot \nabla v \ln \frac{v}{L_1}
   - \int_{\Omega} u^p v \ln \frac{v}{L_1}
   + \int_{\Omega} u^{p+\gamma} \ln \frac{v}{L_1}.
\end{align*}
By using Lemma \ref{le2.10}, we have $\ln \frac{v}{L_1} \le 0$.
Moreover, $-x\ln x \le \frac{1}{e}$ for $x>0$ implies that $-v\ln \frac{v}{L_1} \le \frac{L_1}{e}$.
It follows that
\begin{align*}
   \int_{\Omega} u^p \frac{|\nabla v|^2}{v}
   \le p \int_{\Omega} u^{p-1} |\nabla u| |\nabla v| \ln \frac{L_1}{v}
   + \frac{L_1}{e} \int_{\Omega} u^p.
\end{align*}
For $x, \zeta > 0$, the estimate $\frac{\ln x}{x^\zeta} \le \frac{1}{e\zeta}$ holds.
Taking a fixed $\zeta \in (0, \frac{1}{2})$ and setting $x = \frac{L_1}{v}$,
we have $\ln \frac{L_1}{v} \le \frac{L_1^\zeta}{e\zeta v^\zeta}$.
Therefore, we deduce that
\begin{align}\label{le2.11-aux}
   \int_{\Omega} u^p \frac{|\nabla v|^2}{v}
   \le \frac{p L_1^\zeta}{e\zeta} \int_{\Omega} u^{p-1} |\nabla u| |\nabla v| v^{-\zeta}
   + \frac{L_1}{e} \int_{\Omega} u^p.
\end{align}
For every $\varepsilon_1 > 0$, Young's inequality implies that
\begin{align}\label{6}
    \frac{p L_1^\zeta}{e\zeta} \int_{\Omega} u^{p-1} |\nabla u| |\nabla v| v^{-\zeta}
    \le \varepsilon_1 \int_{\Omega} |\nabla u^{\frac{p}{2}}|^2
    + \frac{p^2 L_1^{2\zeta}}{4\varepsilon_1 e^2 \zeta^2} \int_{\Omega} u^p \frac{|\nabla v|^2}{v^{2\zeta}}.
\end{align}
For every $\varepsilon_2 > 0$, H\"{o}lder's inequality and Young's inequality lead to
\begin{align}\label{7}
    \int_{\Omega} u^p \frac{|\nabla v|^2}{v^{2\zeta}}
    &= \int_{\Omega} \left(u^p \frac{|\nabla v|^2}{v}\right)^{2\zeta}
    \left(u^p |\nabla v|^2\right)^{1-2\zeta} \nonumber\\
    &\le \left(\int_{\Omega} u^p \frac{|\nabla v|^2}{v}\right)^{2\zeta}
    \left(\int_{\Omega} u^p |\nabla v|^2\right)^{1-2\zeta} \nonumber\\
    &\le 2\zeta \varepsilon_2 \int_{\Omega} u^p \frac{|\nabla v|^2}{v}
    + (1-2\zeta) \varepsilon_2^{-\frac{2\zeta}{1-2\zeta}} \int_{\Omega} u^p |\nabla v|^2.
\end{align}
Setting $\varepsilon_2 = \frac{\varepsilon_1 e^2 \zeta}{p^2 L_1^{2\zeta}}$
and substituting \eqref{7} into \eqref{6}, we find that
\begin{align}\label{8}
    \frac{p L_1^\zeta}{e\zeta} \int_{\Omega} u^{p-1} |\nabla u| |\nabla v| v^{-\zeta}
    &\le \varepsilon_1 \int_{\Omega} |\nabla u^{\frac{p}{2}}|^2
    + \frac{1}{2} \int_{\Omega} u^p \frac{|\nabla v|^2}{v} \nonumber\\
    &\quad + \frac{p^2 L_1^{2\zeta}}{4\varepsilon_1 e^2 \zeta^2}
    (1-2\zeta) \left( \frac{\varepsilon_1 e^2 \zeta}{p^2 L_1^{2\zeta}} \right)^{-\frac{2\zeta}{1-2\zeta}}
    \int_{\Omega} u^p |\nabla v|^2.
\end{align}
Substituting \eqref{8} into \eqref{le2.11-aux}, we obtain
\begin{align*}
   \int_{\Omega} u^p \frac{|\nabla v|^2}{v}
   \le 2\varepsilon_1 \int_{\Omega} |\nabla u^{\frac{p}{2}}|^2
   + \frac{p^2 L_1^{2\zeta}}{2\varepsilon_1 e^2 \zeta^2} (1-2\zeta)
   \left( \frac{\varepsilon_1 e^2 \zeta}{p^2 L_1^{2\zeta}} \right)^{-\frac{2\zeta}{1-2\zeta}}
   \int_{\Omega} u^p |\nabla v|^2
   + \frac{2L_1}{e} \int_{\Omega} u^p.
\end{align*}
Setting $\varepsilon_1 = \frac{\tilde{\varepsilon}}{4}$, we infer that
\begin{align}\label{9}
   \int_{\Omega} u^p \frac{|\nabla v|^2}{v}
   \le \frac{\tilde{\varepsilon}}{2} \int_{\Omega} |\nabla u^{\frac{p}{2}}|^2
   + M_\zeta \int_{\Omega} u^p |\nabla v|^2
   + \frac{2L_1}{e} \int_{\Omega} u^p,
\end{align}
where $M_\zeta$ is defined by
\begin{align*}
   M_\zeta = \frac{2p^2 L_1^{2\zeta}}{\tilde{\varepsilon} e^2 \zeta^2} (1-2\zeta)
   \left( \frac{\tilde{\varepsilon} e^2 \zeta}{4p^2 L_1^{2\zeta}} \right)^{-\frac{2\zeta}{1-2\zeta}}.
\end{align*}
Taking $\varepsilon_0 = \frac{\tilde{\varepsilon}}{2M_\zeta}$ in \eqref{5.0} and substituting \eqref{5.0} into \eqref{9},
we conclude that
\begin{align*}
   \int_{\Omega} u^p \frac{|\nabla v|^2}{v}
   \le \tilde{\varepsilon} \int_{\Omega} |\nabla u^{\frac{p}{2}}|^2
   + \frac{2L_1}{e} \int_{\Omega} u^p + C(\tilde{\varepsilon},p).
\end{align*}
$\hfill\Box$

We now combine these estimates to prove Theorem \ref{th3.2}.\\
{\bf Proof of Theorem \ref{th3.2}.}
We first consider the case \(\gamma=1\), and treat the case \(0<\gamma<\frac{2}{N}\) separately in Part II.2 below.
Testing the first equation in \eqref{1.3} by $pu^{p-1}$ for $p\geq2$ and integrating over $\Omega$ gives
\begin{align*}
    \frac{d}{dt}\int_\Omega u^p&= -\frac{4(p-1)}{p}\int_\Omega |\nabla u^{\frac{p}{2}}|^2
    + \chi(p-1)\int_\Omega \nabla u^p\cdot \frac{\nabla v}{v^k}
    + rp\int_\Omega u^{p+\alpha-1}
    - \mu p\int_\Omega u^{p+\alpha-1}\int_\Omega u^\beta \\
    &= -\frac{4(p-1)}{p}\int_\Omega |\nabla u^{\frac{p}{2}}|^2
    - \chi(p-1)\int_\Omega u^p \frac{\Delta v}{v^k}
    + k\chi(p-1)\int_\Omega u^p \frac{|\nabla v|^2}{v^{k+1}} \\
    &\quad + rp\int_\Omega u^{p+\alpha-1} - \mu p\int_\Omega u^{p+\alpha-1}\int_\Omega u^\beta.
\end{align*}
By using the second equation of \eqref{1.3}, we get
\begin{align}\label{4.1}
    \frac{d}{dt}\int_\Omega u^p
    &\leq -\frac{4(p-1)}{p}\int_\Omega |\nabla u^{\frac{p}{2}}|^2
    + \chi(p-1)\int_\Omega \frac{u^{p+1}}{v^k}
    + \chi k(p-1)\int_\Omega u^p \frac{|\nabla v|^2}{v^{k+1}} \nonumber\\
    &\quad + rp\int_\Omega u^{p+\alpha-1} - \mu p\int_\Omega u^{p+\alpha-1}\int_\Omega u^\beta.
\end{align}
By Young's inequality and Lemma \ref{le2.2} with $\lambda_1 = \frac{1}{2p\chi}$, we deduce that
\begin{align}\label{4.2}
    \chi(p-1)\int_\Omega \frac{u^{p+1}}{v^k}
    &= \chi(p-1)\int_\Omega \left( \frac{u^{(p+1)k}}{v^k} \lambda_1^k \right)
    \cdot \left( \lambda_1^{-k} u^{(p+1)(1-k)} \right) \nonumber\\
    &\leq \chi(p-1)\lambda_1 \int_\Omega \frac{u^{p+1}}{v}
    + \chi(p-1)\lambda_1^{-\frac{k}{1-k}} \int_\Omega u^{p+1}\nonumber\\
    &\leq \frac{p-1}{p}\int_\Omega |\nabla u^{\frac{p}{2}}|^2
    + \frac{p-1}{2p}\int_\Omega u^p
    + c_1 \int_\Omega u^{p+1},
\end{align}
where $c_1=(p-1)(2p)^{\frac{k}{1-k}} \chi^{\frac{1}{1-k}}$.
Similarly, set $\lambda_2 = \frac{1}{4p\chi k}$. A further application of Young's inequality and Lemma \ref{le2.2} yields
\begin{align}\label{4.2.1}
    \chi k(p-1)\int_\Omega u^p \frac{|\nabla v|^2}{v^{k+1}}
    &= \chi k(p-1)\int_\Omega \left( u^{pk} \frac{|\nabla v|^{2k}}{v^{2k}} \lambda_2^k \right)
    \cdot \left( \lambda_2^{-k} u^{p(1-k)} \frac{|\nabla v|^{2(1-k)}}{v^{1-k}} \right) \nonumber\\
    &\leq \chi k(p-1)\lambda_2 \int_\Omega u^p \frac{|\nabla v|^2}{v^2}
    + \chi k(p-1)\lambda_2^{-\frac{k}{1-k}} \int_\Omega u^p \frac{|\nabla v|^2}{v}\nonumber\\
    &\leq \frac{p-1}{p}\int_\Omega |\nabla u^{\frac{p}{2}}|^2
    + \frac{p-1}{2p}\int_\Omega u^p
    + c_2 \int_\Omega u^p \frac{|\nabla v|^2}{v},
\end{align}
where $c_2=4^{\frac{k}{1-k}}(\chi k)^{\frac{1}{1-k}}(p-1)p^{\frac{k}{1-k}}$.
Applying H\"{o}lder's inequality, Lemma \ref{le2.3} with $\rho = \kappa = p+1$, and Lemma \ref{le2.4}, we derive
\begin{align*}
    \int_\Omega u^p \frac{|\nabla v|^2}{v}
    &\leq \left( \int_\Omega u^{p+1} \right)^{\frac{p}{p+1}}
    \left( \int_\Omega \frac{|\nabla v|^{2(p+1)}}{v^{p+1}} \right)^{\frac{1}{p+1}} \\
    &\leq \left( \int_\Omega u^{p+1} \right)^{\frac{p}{p+1}}
    \left( c_3 \int_\Omega u^{p+1} + C\int_\Omega v^{p+1} \right)^{\frac{1}{p+1}} \\
    &\leq \left( \int_\Omega u^{p+1} \right)^{\frac{p}{p+1}}
    \big((c_3+C)\int_\Omega u^{p+1}\big)^{\frac{1}{p+1}}\\
    &\leq (c_3+C)^{\frac{1}{p+1}} \int_\Omega u^{p+1},
\end{align*}
where $c_3 = \left( \frac{4p^2}{p-1} \right)^{p+1} \left( \frac{2p}{p-1} \right)^{\frac{p+1}{2}}$.
We then arrive at
\begin{align}\label{4.3}
    \chi k(p-1)\int_{\Omega} u^p \frac{|\nabla v|^2}{v^{k+1}}
    \leq \frac{p-1}{p} \int_{\Omega} |\nabla u^{\frac{p}{2}}|^2
    + \frac{p-1}{2p} \int_{\Omega} u^p + c_2 (c_3 + C)^{\frac{1}{p+1}} \int_{\Omega} u^{p+1}.
\end{align}
Substituting \eqref{4.2} and \eqref{4.3} into \eqref{4.1} yields
\begin{align}\label{4.4}
    \frac{d}{dt} \int_{\Omega} u^p \leq &-\frac{2(p-1)}{p} \int_{\Omega} |\nabla u^{\frac{p}{2}}|^2 + \frac{p-1}{p}
    \int_{\Omega} u^p + M \int_{\Omega} u^{p+1}
    + p \int_{\Omega} u^{p+\alpha-1}(r-\mu \int_{\Omega} u^\beta),
\end{align}
where $M = c_1 + c_2 (c_3 + C)^{\frac{1}{p+1}}$.
\par\medskip
\noindent{\bf I. The superlinear nonlocal damping $\beta>1$.}

\smallskip
\noindent{\bf I.1. The range $1\leq\alpha<2$.}

\smallskip
\noindent{\it The case $N\geq3$.}
Let $p\geq2$ be sufficiently large. For some \(c(\alpha,p,r)>0\), Young's inequality gives
\[
rp\int_\Omega u^{p+\alpha-1}
\leq \int_\Omega u^{p+1}+c(\alpha,p,r),
\quad
\frac{p-1}{p}\int_\Omega u^p
\leq \int_\Omega u^{p+1}+c(p).
\]
Accordingly, \eqref{4.4} takes the form
\begin{align}\label{4.5}
    \frac{d}{dt} \int_{\Omega} u^{p} \leq -\frac{2(p-1)}{p} \int_{\Omega} |\nabla u^{\frac{p}{2}}|^2 + (M+2)
    \int_{\Omega} u^{p+1} - \mu p \int_{\Omega} u^{p+\alpha-1} \int_{\Omega} u^{\beta} + c(\alpha,p).
\end{align}
Set $\omega = u^{\frac{p}{2}}$. In Lemma \ref{le2.5}, take $l = \frac{2p_0}{p}$ and $s = \frac{2(p+1)}{p}$.
Here $p$ and $p_0$ satisfy
\begin{align}\label{4.5.1}
    p > \max\left\{ \frac{2}{q-2}, 1 \right\}, \quad \frac{p}{2} < p_0 < p+1 < \frac{pq}{2}
    \quad \text{and} \quad p_0 > \frac{q}{q-2},
\end{align}
which respectively ensure that $1 \le l < s < q$ and \(\frac{s}{l}<\frac{2}{l}+\frac{2}{N}\).
Let $C_{N_1} = \frac{p-1}{p(M+2)}$ and $C_{N_2} = \frac{1}{M+2}$. Then Lemma \ref{le2.5} entails
\begin{align}\label{4.6}
    \int_{\Omega} u^{p+1} = \| u^{\frac{p}{2}} \|_{L^{\frac{2(p+1)}{p}}(\Omega)}^{\frac{2(p+1)}{p}}
    \leq \frac{p-1}{p(M+2)} \int_{\Omega} |\nabla u^{\frac{p}{2}}|^2 + \frac{1}{M+2} \int_{\Omega} u^{p}
    + c_4 \|u\|_{L^{p_0}(\Omega)}^{\psi}
\end{align}
with $\eta = \frac{\frac{p}{p_0} - \frac{p}{p+1}}{\frac{p}{p_0} - \frac{2}{q}}$, $\theta = \frac{2(1-\eta)s}{2 - \eta s}$,
$\psi = \frac{p\theta}{2} = \frac{p_0(qp - 2p - 2)}{p_0(q-2) - q}$.
Inserting \eqref{4.6} into \eqref{4.5} leads to
\begin{align}\label{4.7}
    \frac{d}{dt} \int_{\Omega} u^{p} \leq -\frac{p-1}{p} \int_{\Omega} |\nabla u^{\frac{p}{2}}|^2
    + \int_{\Omega} u^{p}+c_4 \|u\|_{L^{p_0}(\Omega)}^{\psi}-\mu p\int_{\Omega}u^{p+\alpha-1} \int_{\Omega} u^{\beta}+ c_5.
\end{align}
For $\beta < p_0 < p+\alpha-1$, Lemma \ref{le2.6} implies
\begin{align*}
    \|u\|_{L^{p_0}(\Omega)}^{\psi}
    & \leq \|u\|_{L^{p+\alpha-1}(\Omega)}^{\psi\sigma} \|u\|_{L^{\beta}(\Omega)}^{\psi(1-\sigma)} \\
    & = \left( \|u\|_{L^{p+\alpha-1}(\Omega)}^{p+\alpha-1} \|u\|_{L^{\beta}(\Omega)}^{\beta} \right)^{\frac{\psi \sigma}{p+\alpha-1}}
    \|u\|_{L^{\beta}(\Omega)}^{\psi\left(1-\sigma - \frac{\sigma \beta}{p+\alpha-1}\right)},
\end{align*}
where $\sigma = \frac{\frac{1}{\beta} - \frac{1}{p_0}}{\frac{1}{\beta} - \frac{1}{p+\alpha-1}} \in (0,1)$.
By the arbitrariness of $p_0$, we take $p_0 = \frac{p+\alpha-1+\beta}{2}$ such that
$1-\sigma - \frac{\sigma \beta}{p+\alpha-1} = 0$.
Moreover, since $p_0=\frac{p+\alpha-1+\beta}{2}$,
we obtain
\[
\frac{\psi\sigma}{p+\alpha-1}
=
\frac{\psi}{2p_0}
=
\frac{qp-2p-2}
     {2\left[p_0(q-2)-q\right]}.
\]
Due to the condition \(p_0>\frac{q}{q-2}\), a direct calculation shows that $\frac{\psi}{2p_0}<1$
whenever $\alpha+\beta>3+\frac{2}{q-2}=2+\frac{N}{2}$, which is precisely the corresponding assumption in \(\mathrm{(H_1)}\).
From Young's inequality, we infer that
$$
c_4\|u\|_{L^{p_0}(\Omega)}^\psi \le \frac{\mu p}{3} \int_\Omega u^{p+\alpha-1} \int_\Omega u^\beta + c_6,
$$
where $c_4,c_5,c_6>0$ depend on $p$ and $N$.
Substituting this into \eqref{4.7}, we obtain
\begin{align}\label{4.11.1}
    \frac{d}{dt} \int_\Omega u^p \le -\frac{p-1}{p}\int_\Omega |\nabla u^{\frac{p}{2}}|^2
    + \int_\Omega u^p - \frac{2\mu p}{3} \int_\Omega u^{p+\alpha-1} \int_\Omega u^\beta + c(p,N)
\end{align}
with $c(p,N)>0$.
Lemma \ref{le2.7} and Lemma \ref{le2.1} combined with Young's inequality further show that
\begin{align}\label{4.12}
    \int_{\Omega} u^p
    &= \| u^{\frac{p}{2}}\|_{L^2(\Omega)}^2 \nonumber\\
    &\le C_{GN}\left( \| \nabla u^{\frac{p}{2}} \|_{L^2(\Omega)}^a \| u^{\frac{p}{2}} \|_{L^{\frac{2}{p}}(\Omega)}^{1-a}
    + \| u^{\frac{p}{2}} \|_{L^{\frac{2}{p}}(\Omega)} \right)^2 \nonumber\\
    &\le 2C_{GN}\left( \int_{\Omega} |\nabla u^{\frac{p}{2}}|^2 dx \right)^a \left( \int_{\Omega} u \right)^{(1-a)p}
    + 2C_{GN}\left( \int_{\Omega} u \right)^p \nonumber\\
    &\le \frac{p-1}{4p} \int_{\Omega} |\nabla u^{\frac{p}{2}}|^2  + c(p,N).
\end{align}
Applying \eqref{4.12} to \eqref{4.11.1} gives
\begin{align}\label{4.13}
    \frac{d}{dt} \int_{\Omega} u^p + \int_{\Omega} u^p \le c(p,N),
\end{align}
where the constant $c(p,N)>0$ is independent of $t$,
and hence Gronwall's inequality gives the uniform bound for $\int_\Omega u^p$.

\smallskip
\noindent{\it The case $N=1$.}
For $p\ge 2$, one obtains from Lemma \ref{le2.7}
\begin{align*}
    (M+2)\int_{\Omega} u^{p+1}
    &=(M+2)\|u^{\frac{p}{2}}\|_{L^{\frac{2(p+1)}{p}}(\Omega)}^{\frac{2(p+1)}{p}}\\
    &\le C_{GN}\left(\|\nabla u^{\frac{p}{2}}\|_{L^2(\Omega)}^{a}\|u^{\frac{p}{2}}\|_{L^{\frac{2}{p}}(\Omega)}^{1-a}
    +\|u^{\frac{p}{2}}\|_{L^{\frac{2}{p}}(\Omega)}\right)^{\frac{2(p+1)}{p}}\\
    &\le c(p)\left(\int_{\Omega}|\nabla u^{\frac{p}{2}}|^2\right)^{\frac{a(p+1)}{p}}
    \left(\int_{\Omega}u\right)^{(p+1)(1-a)}
    +c(p)\left(\int_{\Omega}u\right)^{p+1}.
\end{align*}
Since $N=1$, we have $\frac{a(p+1)}{p}< 1$.
Combining Lemma \ref{le2.1} and Young's inequality, we obtain
$$
(M+2)\int_{\Omega}u^{p+1}
\le \frac{p-1}{p}\int_{\Omega}|\nabla u^{\frac{p}{2}}|^2
+c(p).
$$
Using this estimate together with \eqref{4.12} in \eqref{4.5} yields
$$
\frac{d}{dt}\int_{\Omega}u^p+\int_{\Omega}u^p\le c(p),
$$
where $c(p)>0$ is independent of $t$, so the uniform $L^p$-bound follows from Gronwall's inequality.

\smallskip
\noindent{\it The case $N=2$.}
Let \(p\geq2\) be sufficiently large. Set $p_0:=\frac{p+\alpha-1+\beta}{2}$, $l:=\frac{2p_0}{p}$ and $s:=\frac{2(p+1)}{p}$.
This choice ensures that $\beta<p_0<p+\alpha-1$, $1\leq l<s$ and $\frac{s}{l}<\frac{2}{l}+1$.
An application of Lemma \ref{le2.5} to \(\omega=u^{\frac p2}\) yields
\begin{align}\label{4.6-N2}
(M+2)\int_\Omega u^{p+1}
&\leq
\frac{p-1}{p}
\int_\Omega|\nabla u^{\frac p2}|^2
+\int_\Omega u^p
+c_7\|u\|_{L^{p_0}(\Omega)}^\psi,
\end{align}
where $\psi=\frac{p\theta}{2}=\frac{pp_0}{p_0-1}$.
For $p_0=\frac{p+\alpha-1+\beta}{2}$,
Lemma \ref{le2.6} gives
\begin{align*}
\|u\|_{L^{p_0}(\Omega)}^\psi
&\leq
\|u\|_{L^{p+\alpha-1}(\Omega)}^{
\frac{\psi(p+\alpha-1)}{2p_0}}
\|u\|_{L^\beta(\Omega)}^{
\frac{\psi\beta}{2p_0}}\\
&=
\left(
\int_\Omega u^{p+\alpha-1}
\int_\Omega u^\beta
\right)^{\frac{\psi}{2p_0}}.
\end{align*}
A direct calculation confirms $\frac{\psi}{2p_0}<1$ under $\alpha+\beta>3$.
We thus obtain from Young's inequality
\begin{align}\label{4.7-N2}
c_7\|u\|_{L^{p_0}(\Omega)}^\psi
\leq
\frac{\mu p}{2}
\int_\Omega u^{p+\alpha-1}
\int_\Omega u^\beta
+c_8,
\end{align}
where $c_7,c_8>0$ depend on $p$.
From \eqref{4.5}, \eqref{4.6-N2} and \eqref{4.7-N2}, we derive for some $c(p) > 0$ that
\[
\frac{d}{dt}\int_\Omega u^p
\leq
-\frac{p-1}{p}
\int_\Omega|\nabla u^{\frac p2}|^2
+\int_\Omega u^p
-\frac{\mu p}{2}
\int_\Omega u^{p+\alpha-1}
\int_\Omega u^\beta
+c(p).
\]
Combining \eqref{4.12} with Young's inequality, we arrive at
\[
\frac{d}{dt}\int_\Omega u^p
+\int_\Omega u^p
\leq c(p),
\]
where $c(p)>0$ is independent of $t$.
An application of Gronwall's inequality completes the $L^p$-estimate for every sufficiently large $p$,
and H\"older's inequality covers the remaining exponents $p\ge 2$.

\smallskip
\noindent{\bf I.2. The range $\alpha\geq2$.}

For every $N\geq1$, Young's inequality yields constants $K_p>0$ and $c(p)>0$ independent of $t$ such that \eqref{4.4} becomes
\begin{align}\label{4.14}
\frac{d}{dt}\int_\Omega u^p
&\leq-\frac{2(p-1)}{p}\int_\Omega|\nabla u^{\frac p2}|^2
+K_p\int_\Omega u^{p+\alpha-1} -\mu p\int_\Omega u^{p+\alpha-1}
\int_\Omega u^\beta+c(p).
\end{align}
\smallskip
\noindent{\it The case $N\geq3$.}
Let \(p\geq2\) satisfy $p>\frac{(N-2)(\alpha-1)}{2}$ and choose $1 < p_1 < p+\alpha-1$.
Applying H\"{o}lder's inequality and the Sobolev embedding theorem, we prove that
\begin{align*}
    \int_{\Omega} u^{p+\alpha-1}
    &= \int_{\Omega} u^{\lambda_3 \frac{p}{2} \frac{2(p+\alpha-1)}{p}}
    u^{(1-\lambda_3) \frac{p}{2} \frac{2(p+\alpha-1)}{p}} \\
    &\le \| u^{\frac{p}{2}} \|_{L^q(\Omega)}^{\frac{2\lambda_3(p+\alpha-1)}{p}}
    \| u^{\frac{p}{2}} \|_{L^{\frac{2p_1}{p}}(\Omega)}^{\frac{2(1-\lambda_3)(p+\alpha-1)}{p}} \\
    &\le c_{9} \left(\| \nabla u^{\frac{p}{2}} \|_{L^2(\Omega)}^{\lambda_3}
    \| u^{\frac{p}{2}}\|_{L^{\frac{2p_1}{p}}(\Omega)}^{1-\lambda_3}
    + \| u^{\frac{p}{2}} \|_{L^2(\Omega)}^{\lambda_3}
    \| u^{\frac{p}{2}} \|_{L^{\frac{2p_1}{p}}(\Omega)}^{1-\lambda_3}\right)^{\frac{2(p+\alpha-1)}{p}} \\
    &\le c_{10} \Bigg[\left( \int_{\Omega} |\nabla u^{\frac{p}{2}}|^2 \right)^{\frac{\lambda_3(p+\alpha-1)}{p}}
    \left( \int_{\Omega} u^{p_1} \right)^{\frac{(1-\lambda_3)(p+\alpha-1)}{p_1}}
    +\left( \int_{\Omega} u^p \right)^{\frac{\lambda_3(p+\alpha-1)}{p}}
    \left( \int_{\Omega} u^{p_1} \right)^{\frac{(1-\lambda_3)(p+\alpha-1)}{p_1}}\Bigg],
\end{align*}
where the constants $c_9,c_{10}>0$ depend on $p$ and $N$.
It follows from H\"{o}lder's inequality that
$$
\lambda_3 = \frac{\frac{Np}{2p_1} - \frac{Np}{2(p+\alpha-1)}}{\frac{Np}{2p_1} + 1 - \frac{N}{2}} \in (0,1).
$$
In addition, choose $ p_1 > \frac{N(\alpha-1)}{2} $ such that $ \frac{\lambda_3(p+\alpha-1)}{p} < 1 $.
Young's inequality then leads to
\begin{align}\label{4.15}
    K_p\int_{\Omega}u^{p+\alpha-1}
    \le \frac{p-1}{p}\int_{\Omega}|\nabla u^{\frac{p}{2}}|^2 + c_{11}\left(\int_{\Omega}u^{p_1}\right)^{\psi}
     + c_{10}\left(\int_{\Omega}u^{p}\right)^{\frac{\lambda_3(p+\alpha-1)}{p}}
    \left(\int_{\Omega}u^{p_1}\right)^{\frac{(1-\lambda_3)(p+\alpha-1)}{p_1}},
\end{align}
where $\psi = \frac{(1-\lambda_3)(p+\alpha-1)}{p_1} \cdot \frac{p}{p-\lambda_3(p+\alpha-1)}$.
Take $\beta < p_1 < p+\alpha-1$. In view of Lemma \ref{le2.6}, we infer that
\begin{align}\label{4.15.1}
    \|u\|_{L^{p_1}(\Omega)} \le \|u\|_{L^{p+\alpha-1}(\Omega)}^{\sigma} \|u\|_{L^{\beta}(\Omega)}^{1-\sigma},
\end{align}
where $\sigma = \frac{\frac{1}{\beta} - \frac{1}{p_1}}{\frac{1}{\beta} - \frac{1}{p+\alpha-1}} \in (0,1)$.
Thus, we derive that
\begin{align*}
    \left( \int_{\Omega} u^{p_1} \right)^{\psi}
    &= \|u\|_{L^{p_1}(\Omega)}^{\psi p_1}
    \le \|u\|_{L^{p+\alpha-1}(\Omega)}^{\psi p_1 \sigma} \|u\|_{L^{\beta}(\Omega)}^{\psi p_1(1-\sigma)} \\
    &= \left( \|u\|_{L^{p+\alpha-1}(\Omega)}^{p+\alpha-1} \|u\|_{L^{\beta}(\Omega)}^{\beta} \right)^{\frac{\psi p_1 \sigma}{p+\alpha-1}}
    \|u\|_{L^{\beta}(\Omega)}^{\psi p_1(1-\sigma) - \frac{\psi p_1 \sigma \beta}{p+\alpha-1}}.
\end{align*}
Fix $p_1=\frac{p+\alpha-1+\beta}{2}$,
which ensures $p_1(1-\sigma)-\frac{p_1\sigma\beta}{p+\alpha-1}=0$,
equivalently $1-\sigma-\frac{\sigma\beta}{p+\alpha-1}=0$.
Moreover, the assumption $ 2 \le \alpha < 1 + \frac{2\beta}{N} $ guarantees that $\frac{\psi p_1 \sigma}{p+\alpha-1} < 1$.
Young's inequality therefore entails
\begin{align}\label{4.16}
    c_{11} \left( \int_{\Omega} u^{p_1} \right)^{\psi}
    \le \frac{\mu p}{4} \int_{\Omega} u^{p+\alpha-1} \int_{\Omega} u^{\beta}  + c_{12}.
\end{align}
Substituting \eqref{4.16} into \eqref{4.15} and combining with \eqref{4.14} yields
\begin{align}\label{4.17}
    \frac{d}{dt} \int_{\Omega} u^p
    \le &-\frac{p-1}{p} \int_{\Omega} |\nabla u^{\frac{p}{2}}|^2
    - \frac{3\mu p}{4} \int_{\Omega} u^{p+\alpha-1} \int_{\Omega} u^{\beta}\nonumber\\
    &+ c_{10} \left( \int_{\Omega} u^p \right)^{\frac{\lambda_3(p+\alpha-1)}{p}}
    \left( \int_{\Omega} u^{p_1} \right)^{\frac{(1-\lambda_3)(p+\alpha-1)}{p_1}}+ c_{13}.
\end{align}
In addition, H\"{o}lder's inequality implies
$$
\|u^{\beta}\|_{L^1(\Omega)}
\le \|u^{\beta}\|_{L^{\frac{p+\alpha-1}{\beta}}(\Omega)}
\|1\|_{L^{\frac{p+\alpha-1}{p+\alpha-1-\beta}}(\Omega)}.
$$
It follows that
\begin{align}\label{4.18}
    \frac{\mu p}{4} \left( \int_{\Omega} u^{\beta}  \right)^{\frac{p+\alpha-1+\beta}{\beta}}
    \le \frac{\mu p}{4} \int_{\Omega} u^{p+\alpha-1}  \int_{\Omega} u^{\beta}.
\end{align}
Adding \eqref{4.18} to \eqref{4.17} and using \eqref{4.15.1}, we deduce that
\begin{align}\label{4.19}
\quad\frac{d}{dt}\int_{\Omega}u^p+\frac{\mu p}{4}
    \left(\int_{\Omega}u^{\beta}\right)^{\frac{p+\alpha-1+\beta}{\beta}}
    \le &-\frac{p-1}{p}\int_{\Omega}|\nabla u^{\frac{p}{2}}|^2
    -\frac{\mu p}{2}\int_{\Omega}u^{p+\alpha-1}\int_{\Omega}u^{\beta}\nonumber\\
     &+ c_{10}\left(\int_{\Omega}u^p\right)^{\frac{\lambda_3(p+\alpha-1)}{p}}
    \left(\|u\|_{L^{p+\alpha-1}(\Omega)}^{\sigma}\|u\|_{L^{\beta}(\Omega)}^{1-\sigma}\right)^{(1-\lambda_3)(p+\alpha-1)}+c_{13}.
\end{align}
Assume first that $\beta\geq2$ and take $p=\beta$. Via $1-\sigma-\frac{\sigma\beta}{p+\alpha-1}=0$
and Young's inequality with indices
$\left(\frac{1}{(1-\lambda_3)\sigma},\ \frac{1}{1-\sigma(1-\lambda_3)}\right)$,
we infer that
\begin{align*}
    &\quad c_{10}\left(\int_{\Omega}u^{\beta}\right)^{\frac{\lambda_3(\beta+\alpha-1)}{\beta}}
    \left(\|u\|_{L^{\beta+\alpha-1}(\Omega)}^{\sigma}\|u\|_{L^{\beta}(\Omega)}^{1-\sigma}\right)^{(1-\lambda_3)(\beta+\alpha-1)}\\
    &= c_{10}\left(\int_{\Omega}u^{\beta}\right)^{\frac{\lambda_3(\beta+\alpha-1)}{\beta}}
    \left(\|u\|_{L^{\beta+\alpha-1}(\Omega)}^{\beta+\alpha-1}\|u\|_{L^{\beta}(\Omega)}^{\beta}\right)^{\sigma(1-\lambda_3)} \\
    &\le \frac{\mu\beta}{4}\int_{\Omega}u^{\beta+\alpha-1}\int_{\Omega}u^{\beta}
    + c_{14}\left(\int_{\Omega}u^{\beta}\right)^{\frac{\lambda_3(\beta+\alpha-1)}{\beta}\cdot\frac{1}{1-\sigma(1-\lambda_3)}}.
\end{align*}
A direct computation verifies that $\frac{\lambda_3(\beta+\alpha-1)}{1-\sigma(1-\lambda_3)} < 2\beta+\alpha-1$.
From Young's inequality, one infers
\begin{align}\label{4.19.1}
    &\quad c_{10}\left(\int_{\Omega}u^{\beta}\right)^{\frac{\lambda_3(\beta+\alpha-1)}{\beta}}
    \left(\|u\|_{L^{\beta+\alpha-1}(\Omega)}^{\sigma}\|u\|_{L^{\beta}(\Omega)}^{1-\sigma}\right)^{(1-\lambda_3)(\beta+\alpha-1)}\nonumber\\
    &\le \frac{\mu\beta}{4}\int_{\Omega}u^{\beta+\alpha-1}\int_{\Omega}u^{\beta}
    + \frac{\mu\beta}{8}\left(\int_{\Omega}u^{\beta}\right)^{\frac{2\beta+\alpha-1}{\beta}} + c_{15}.
\end{align}
Here $c_{11},\ldots,c_{15}>0$ depend on $p$ and $N$.
For $p=\beta$, inserting \eqref{4.19.1} into \eqref{4.19}, we find that
\begin{align*}
    \frac{d}{dt}\int_{\Omega}u^{\beta}
    + \frac{\mu\beta}{8}\left(\int_{\Omega}u^{\beta}dx\right)^{\frac{2\beta+\alpha-1}{\beta}}
    \le -\frac{\beta-1}{\beta}\int_{\Omega}|\nabla u^{\frac{\beta}{2}}|^2
    - \frac{\mu\beta}{4}\int_{\Omega}u^{\beta+\alpha-1}\int_{\Omega}u^{\beta} + c(\beta,N).
\end{align*}
It follows from Lemma \ref{le2.8} that
$$
\int_{\Omega}u^{\beta}dx \le \max\left\{ \|u_0\|_{L^{\beta}(\Omega)}^{\beta},\ c(\beta,N) \right\},
$$
where $c(\beta,N)>0$.
Let \(1<\beta<2\). The assumptions \(N\geq3\) and
\(2\leq\alpha<1+\frac{2\beta}{N}\) imply \(N<2\beta<4\) and hence \(N=3\).
Taking \(p=2\) and \(p_1=\frac{\alpha+1+\beta}{2}\) in \eqref{4.19},
set $\sigma_0:=\frac{\frac1\beta-\frac12}{\frac1\beta-\frac1{\alpha+1}}$
and $\sigma_1:=\frac{\alpha+1}{\alpha+1+\beta}$.
Since \(2<p_1\) and \(0<\lambda_3<1\), it follows that
\[
0<\sigma_0\lambda_3+\sigma_1(1-\lambda_3)<\sigma_1
=\frac{\alpha+1}{\alpha+1+\beta}.
\]
For some $c>0$, applying Lemma \ref{le2.6} and Young's inequality, we derive
\begin{align}\label{4.19.1.0}
\left(\int_\Omega u^2\right)^{\frac{\lambda_3(\alpha+1)}{2}}
\left(\|u\|_{L^{\alpha+1}(\Omega)}^{\sigma_1}\|u\|_{L^\beta(\Omega)}^{1-\sigma_1}
\right)^{(1-\lambda_3)(\alpha+1)}
&\leq\|u\|_{L^{\alpha+1}(\Omega)}^{
[\sigma_0\lambda_3+\sigma_1(1-\lambda_3)](\alpha+1)}
\|u\|_{L^\beta(\Omega)}^{[(1-\sigma_0)\lambda_3
 +(1-\sigma_1)(1-\lambda_3)](\alpha+1)}\nonumber\\
&\leq\frac{\mu}{2}\int_\Omega u^{\alpha+1}
\int_\Omega u^\beta+\frac{\mu}{8}
\left(\int_\Omega u^\beta\right)^{
\frac{\alpha+1+\beta}{\beta}}+c.
\end{align}
Inserting \eqref{4.19.1.0} into \eqref{4.19} with \(p=2\), we find \(c_{16},c_{17},c_{18}>0\) such that
\[
\frac{d}{dt}\int_\Omega u^2+c_{16}\int_\Omega|\nabla u|^2+c_{17}
\left(\int_\Omega u^\beta\right)^{\frac{\alpha+1+\beta}{\beta}}\leq c_{18}.
\]
Applying Lemma \ref{le2.7} with \(r=m_0=2\) and \(s_0=z=1\), we use Lemma \ref{le2.1} and Young's inequality to derive
\[
\int_\Omega u^2\leq\frac{c_{16}}{2}\int_\Omega|\nabla u|^2+c_{19}.
\]
We thus obtain
\[
\frac{d}{dt}\int_\Omega u^2+\int_\Omega u^2\leq c,
\]
where $c>0$ is independent of $t$.
Gronwall's inequality combined with H\"older's inequality yields
\[
\sup_{0<t<T_{\max}}\int_\Omega u^\beta\leq c(\beta,N).
\]
Combining the cases \(1<\beta<2\) and \(\beta\geq2\), we conclude that
\begin{align}\label{4.19.2}
\sup_{0<t<T_{\max}}\int_\Omega u^\beta\leq c(\beta,N).
\end{align}

We next establish the \(L^p\)-estimate for sufficiently large \(p>\beta\).
By Lemma \ref{le2.7} and \eqref{4.19.2}, we obtain
\begin{align*}
\int_\Omega u^{p+\alpha-1}&=\|u^{\frac p2}\|_{L^{\frac{2(p+\alpha-1)}{p}}(\Omega)
}^{\frac{2(p+\alpha-1)}{p}}
\leq c_{20}\left(\int_\Omega|\nabla u^{\frac p2}|^2\right)^a+c_{21},
\end{align*}
where $a=\frac{p+\alpha-1}{p}\frac{\frac{p}{2\beta}-\frac{p}{2(p+\alpha-1)}
}{\frac{p}{2\beta}+\frac1N-\frac12}$, and $c_{20},c_{21}>0$ depend on $p$ and $N$.
A direct calculation shows $a<1$ provided $\alpha<1+\frac{2\beta}{N}$.
We then apply Young's inequality to obtain
\[
K_p\int_\Omega u^{p+\alpha-1}\leq
\frac{p-1}{p}\int_\Omega|\nabla u^{\frac p2}|^2+c(p,N).
\]
Substituting this estimate into \eqref{4.14} yields
\begin{align}\label{4.20}
\frac{d}{dt}\int_\Omega u^p+\frac{p-1}{p}\int_\Omega
|\nabla u^{\frac p2}|^2\leq c(p,N).
\end{align}
By Lemma \ref{le2.7}, \eqref{4.19.2} and Young's inequality, we have
\[
\int_\Omega u^p\leq\frac{p-1}{2p}
\int_\Omega|\nabla u^{\frac p2}|^2+c(p,N).
\]
We combine this estimate with \eqref{4.20} and apply Gronwall's inequality to obtain
\[
\sup_{0<t<T_{\max}}\int_\Omega u^p\leq c(p,N),
\]
where $c(p,N)>0$ is independent of $t$.
H\"older's inequality covers the remaining exponents.

For $N\in\{1,2\}$, choose sufficiently large $p\geq2$ and set
$p_2=\frac{p+\alpha-1+\beta}{2}$.
We thus have $\beta<p_2<p+\alpha-1$ and ensure that
$p_2>\max\left\{\frac p2,\frac{\alpha-1}{2}\right\}$ if $N=1$,
and $p_2>\max\left\{\frac p2,\alpha-1\right\}$ if $N=2$.
Set $\omega=u^{\frac p2}$, $l=\frac{2p_2}{p}$ and $s=\frac{2(p+\alpha-1)}{p}$.
Taking $C_{N_1}=\frac{p-1}{pK_p}$ and $C_{N_2}=\frac1{K_p}$ in Lemma \ref{le2.5}, we arrive at
\begin{align}\label{4.21}
K_p\int_\Omega u^{p+\alpha-1}
&\leq\frac{p-1}{p}\int_\Omega|\nabla u^{\frac p2}|^2
+\int_\Omega u^p+c_{22}\left(\int_\Omega u^{p_2}\right)^{\frac{p\theta}{2p_2}},
\end{align}
where $\eta=\frac{\frac{p}{p_2}-\frac{p}{p+\alpha-1}}{\frac{p}{p_2}+\frac{2}{N}-1}$ and $\theta=\frac{2(1-\eta)s}{2-\eta s}$.
By Lemma \ref{le2.6}, we get
\begin{align*}
\|u\|_{L^{p_2}(\Omega)}
\leq
\|u\|_{L^\beta(\Omega)}^\sigma
\|u\|_{L^{p+\alpha-1}(\Omega)}^{1-\sigma},
\end{align*}
where $\sigma=\frac{\frac1{p_2}-\frac1{p+\alpha-1}}
{\frac1\beta-\frac1{p+\alpha-1}}=\frac{\beta}{p+\alpha-1+\beta}$
and $1-\sigma-\frac{p+\alpha-1}{\beta}\sigma=0$.
Thus, we have
\begin{align*}
\left(\int_\Omega u^{p_2}\right)^{\frac{p\theta}{2p_2}}
&=\|u\|_{L^{p_2}(\Omega)}^{\frac{p\theta}{2}}
\leq
\left(\int_\Omega u^{p+\alpha-1}\int_\Omega u^\beta\right)^{
\frac{p\theta\sigma}{2\beta}}.
\end{align*}
A direct calculation shows that \(\frac{p\theta\sigma}{2\beta}=\frac{2p+\alpha-1}{2(p+\beta)}\) for \(N=1\),
and \(\frac{p\theta\sigma}{2\beta}=\frac{p}{p+\beta-\alpha+1}\) for \(N=2\).
Since \(\alpha<1+\frac{2\beta}{N}\),
Young's inequality yields
\begin{align}\label{4.22}
c_{22}
\left(\int_\Omega u^{p_2}\right)^{
\frac{p\theta}{2p_2}}
&\leq\frac{\mu p}{2}
\int_\Omega u^{p+\alpha-1}
\int_\Omega u^\beta+c_{23},
\end{align}
where $c_{22},c_{23}>0$ depend on $p$ and $N$.
Combining \eqref{4.14}, \eqref{4.21} and \eqref{4.22}, we derive
\begin{align*}
\frac{d}{dt}\int_\Omega u^p
&\leq-\frac{p-1}{p}\int_\Omega|\nabla u^{\frac p2}|^2
+\int_\Omega u^p-\frac{\mu p}{2}\int_\Omega u^{p+\alpha-1}
\int_\Omega u^\beta+c(p,N).
\end{align*}
We use \eqref{4.12} and Gronwall's inequality to conclude that
\begin{align*}
\sup_{0<t<T_{\max}}\int_\Omega u^p\leq c(p,N),
\end{align*}
where $c(p,N)>0$ is independent of $t$.
H\"older's inequality covers all remaining $p\geq2$.

\par\medskip
\noindent{\bf II. The linear nonlocal damping $\beta=1$.}

\smallskip
\noindent{\bf II.1. The one-dimensional setting $N=1$.}\par
Let $m(t)=\int_{\Omega}u(t)dx$ and $m_0=\int_{\Omega}u_0(x)dx$.
We rewrite \eqref{4.4} as
\begin{align}\label{4.26}
    \frac{d}{dt}\int_{\Omega}u^p + \int_{\Omega}u^p+ \frac{2(p-1)}{p}\int_{\Omega}|\nabla u^{\frac{p}{2}}|^2
    \le \frac{2p-1}{p}\int_{\Omega}u^p+ M\int_{\Omega}u^{p+1}+ (r-\mu m(t))p\int_{\Omega}u^{p+\alpha-1}.
\end{align}
\smallskip
\noindent{\it (a) The subcritical initial mass $m_0 < \frac{r}{\mu}$.}

By Lemma \ref{le2.1}, we have $0 < r-\mu m(t) \le r-\mu m_0$.
For $1 \le \alpha < 2$, Young's inequality provides $c_{24},c_{25}>0$ depending on $p$ such that
$$
\int_{\Omega}u^p\le \int_{\Omega}u^{p+1} + c_{24}, \quad
\int_{\Omega}u^{p+\alpha-1}\le \int_{\Omega}u^{p+1} + c_{25}.
$$
It then follows that
\begin{align}\label{4.27}
    \frac{d}{dt}\int_{\Omega}u^p+ \int_{\Omega}u^p+ \frac{2(p-1)}{p}\int_{\Omega}|\nabla u^{\frac{p}{2}}|^2
    \le (M+2+rp-\mu m_0p)\int_{\Omega}u^{p+1} + c(p).
\end{align}
For $p > 1$, we apply Lemma \ref{le2.7} and derive
\begin{align*}
    \int_{\Omega} u^{p+1}=\|u^{\frac{p}{2}}\|_{L^{\frac{2(p+1)}{p}}(\Omega)}^{\frac{2(p+1)}{p}}
&\le C_{GN}\left( \|\nabla u^{\frac{p}{2}}\|_{L^2(\Omega)}^{a} \|u^{\frac{p}{2}}\|_{L^{\frac{2}{p}}(\Omega)}^{1-a}
+\|u^{\frac{p}{2}}\|_{L^{\frac{2}{p}}(\Omega)} \right)^{\frac{2(p+1)}{p}} \\
&\le c(p)\left(\int_{\Omega} |\nabla u^{\frac{p}{2}}|^2\right)^{\frac{a(p+1)}{p}}
\left(\int_{\Omega} u\right)^{(p+1)(1-a)}
+ c(p)\left(\int_{\Omega} u\right)^{p+1},
\end{align*}
where $a = \frac{\frac{p}{2} - \frac{p}{2(p+1)}}{\frac{p}{2} + \frac{1}{N} - \frac{1}{2}}$.
Since $N =1$, we have $\frac{a(p+1)}{p} < 1$. Combining Young's inequality with Lemma \ref{le2.1}, we obtain
\begin{align}\label{4.28}
    (M+2+rp-\mu m_0p)\int_{\Omega} u^{p+1}\le \frac{p-1}{p}\int_{\Omega} |\nabla u^{\frac{p}{2}}|^2 + c(p).
\end{align}
Substituting \eqref{4.28} into \eqref{4.27}, we get
$$
\frac{d}{dt}\int_{\Omega} u^p + \int_{\Omega} u^p \le c(p),
$$
where $c(p)>0$ is independent of $t$.

For $\alpha \ge 2$, by Young's inequality, we have
$$
\int_{\Omega} u^p \le \int_{\Omega} u^{p+\alpha-1} + c(\alpha,p), \quad
\int_{\Omega} u^{p+1} \le \int_{\Omega} u^{p+\alpha-1} + c(\alpha,p).
$$
Then \eqref{4.26} can be rewritten as
\begin{align}\label{4.29}
\frac{d}{dt}\int_{\Omega} u^p
+ \int_{\Omega} u^p
+ \frac{2(p-1)}{p} \int_{\Omega} |\nabla u^{\frac{p}{2}}|^2
\le (M+2+r p-\mu m_0p) \int_{\Omega} u^{p+\alpha-1} + c(p).
\end{align}
Similarly, it follows from Lemma \ref{le2.7} that
\begin{align*}
    \int_{\Omega} u^{p+\alpha-1}
    &= \|u^{\frac{p}{2}}\|_{L^{\frac{2(p+\alpha-1)}{p}}(\Omega)}^{\frac{2(p+\alpha-1)}{p}} \\
    &\le C_{GN} \left( \|\nabla u^{\frac{p}{2}}\|_{L^2(\Omega)}^{a} \|u^{\frac{p}{2}}\|_{L^{\frac{2}{p}}(\Omega)}^{1-a}
    + \|u^{\frac{p}{2}}\|_{L^{\frac{2}{p}}(\Omega)} \right)^{\frac{2(p+\alpha-1)}{p}} \\
    &\le c(p) \left( \int_{\Omega} |\nabla u^{\frac{p}{2}}|^2 \right)^{\frac{a(p+\alpha-1)}{p}}
    \left( \int_{\Omega} u \right)^{(p+\alpha-1)(1-a)}
    + c(p)\left( \int_{\Omega} u \right)^{p+\alpha-1},
\end{align*}
where $a = \frac{\frac{p}{2} - \frac{p}{2(p+\alpha-1)}}{\frac{p}{2} + \frac{1}{N} - \frac{1}{2}}$.
Since $\alpha < 1 + \frac{2}{N}$, we have $\frac{a(p+\alpha-1)}{p} < 1$.
Then by Young's inequality and Lemma \ref{le2.1}, we deduce
\begin{align}\label{4.30}
    (M+2+r p-\mu m_0p)\int_{\Omega}u^{p+\alpha-1}\le \frac{p-1}{p}\int_{\Omega}|\nabla u^{\frac{p}{2}}|^2 + c(p).
\end{align}
For some $c(p)>0$ independent of $t$, substituting \eqref{4.30} into \eqref{4.29} yields
$$
\frac{d}{dt}\int_{\Omega}u^p + \int_{\Omega}u^p \le c(p).
$$

\smallskip
\noindent{\it (b) The critical or supercritical initial mass $m_0 \geq \frac{r}{\mu}$.}

By Lemma \ref{le2.1}, we know that $\frac{r}{\mu} \le m(t) \le m_0$.
Then we rewrite \eqref{4.26} as
\begin{align}\label{4.31}
    \frac{d}{dt}\int_{\Omega}u^p+ \int_{\Omega}u^p+ \frac{2(p-1)}{p}\int_{\Omega}|\nabla u^{\frac{p}{2}}|^2
    \le \frac{2p-1}{p}\int_{\Omega}u^p+ M\int_{\Omega}u^{p+1}.
\end{align}
We proceed as in \eqref{4.28} to obtain
\[
\frac{d}{dt}\int_{\Omega}u^p + \int_{\Omega}u^p \le c(p),
\]
where $c(p)>0$ is independent of $t$.
This completes the one-dimensional case.

\smallskip
\noindent{\bf II.2. The higher-dimensional setting $N\geq2$.}\par
Here \(0<\gamma<\frac{2}{N}\).
Repeating the derivation of
\eqref{4.1}-\eqref{4.2} with $u^{p+\gamma}$ in place of $u^{p+1}$, we conclude that
\begin{align}\label{4.32}
    \frac{d}{dt} \int_{\Omega} u^p
    &\le -\frac{3(p-1)}{p} \int_{\Omega} |\nabla u^{\frac{p}{2}}|^2
    + \chi k(p-1) \int_{\Omega} u^p \frac{|\nabla v|^2}{v^{1+k}} \nonumber\\
    &\quad + \frac{p-1}{2p} \int_{\Omega} u^p
    + c_1 \int_{\Omega} u^{p+\gamma}
    + (r-\mu m(t)) p \int_{\Omega} u^{p+\alpha-1},
\end{align}
where $c_1 = (p-1)(2p)^{\frac{k}{1-k}} \chi^{\frac{1}{1-k}}$.
Combining \eqref{4.32} with \eqref{4.2.1} and Lemma \ref{le2.11}, we get
\begin{align}\label{4.33}
    \chi k(p-1) \int_{\Omega} u^p \frac{|\nabla v|^2}{v^{1+k}}
    \le \frac{3(p-1)}{2p} \int_{\Omega} |\nabla u^{\frac{p}{2}}|^2
    + \left( \frac{p-1}{2p} + \frac{2L_1}{e} c_2 \right) \int_{\Omega} u^p + c_{28}.
\end{align}
Substituting \eqref{4.33} into \eqref{4.32} leads to
\begin{align}\label{4.34}
    \frac{d}{dt} \int_{\Omega} u^p
    &\le -\frac{3(p-1)}{2p} \int_{\Omega} |\nabla u^{\frac{p}{2}}|^2
    + \left( \frac{p-1}{p} + \frac{2L_1}{e} c_2 \right) \int_{\Omega} u^p \nonumber\\
    &\quad + c_1 \int_{\Omega} u^{p+\gamma}
    + (r-\mu m(t)) p \int_{\Omega} u^{p+\alpha-1} + c_{28},
\end{align}
where $c_{28}=c_{28}(p,N)>0$ is independent of $t$.

\smallskip
\noindent{\it (a) The subcritical initial mass $m_0 < \frac{r}{\mu}$.}

By Lemma \ref{le2.1}, we have $0<r-\mu m(t)\leq r-\mu m_0$.
Set $\phi:=\max\{\gamma,\alpha-1\}$.
Moreover, Young's inequality gives
$$
\int_\Omega u^p
\leq
\int_\Omega u^{p+\phi}+c(p),
\quad
\int_\Omega u^{p+\gamma}
\leq
\int_\Omega u^{p+\phi}+c(p)
\quad \text{and}\quad
\int_\Omega u^{p+\alpha-1}
\leq
\int_\Omega u^{p+\phi}+c(p).
$$
Thus, it follows from \eqref{4.34} that
\begin{align}\label{4.35}
    \frac{d}{dt} \int_{\Omega} u^p
    \le -\frac{3(p-1)}{2p} \int_{\Omega} |\nabla u^{\frac{p}{2}}|^2
    + c_{29} \int_{\Omega} u^{p+\phi} + c_{30},
\end{align}
where $c_{29}:=\frac{p-1}{p}+\frac{2L_1}{e}c_2+c_1+rp-\mu m_0p>0$ and $c_{30}=c_{30}(p,N)>0$.
Proceeding as in \eqref{4} and using Lemmas \ref{le2.7} and \ref{le2.1}, we obtain
\begin{align}\label{4.35.1}
    \int_{\Omega} u^{p+\phi}
    = \| u^{\frac{p}{2}} \|_{L^{\frac{2(p+\phi)}{p}}(\Omega)}^{\frac{2(p+\phi)}{p}}
    \le c_{31} \left( \int_{\Omega} |\nabla u^{\frac{p}{2}}|^2 \right)^{\frac{p+\phi-1}{p-1+\frac{2}{N}}} + c_{32}.
\end{align}
Since $\phi < \frac{2}{N}$, Young's inequality provides
\begin{align}\label{4.36}
    \int_{\Omega} u^{p+\phi}
    \le \frac{p-1}{2c_{29} p} \int_{\Omega} |\nabla u^{\frac{p}{2}}|^2 + c_{33},
\end{align}
where $c_{31},c_{32},c_{33}>0$ depend on $p$ and $N$.
Substituting \eqref{4.36} into \eqref{4.35}, we infer that
\begin{align}\label{4.37}
    \frac{d}{dt} \int_{\Omega} u^p
    \le -\frac{p-1}{p} \int_{\Omega} |\nabla u^{\frac{p}{2}}|^2 + c(p,N).
\end{align}
Combining \eqref{4.12} and \eqref{4.37} and applying Gronwall's inequality yields
\begin{align*}
\int_{\Omega}u^p\leq c(p,N).
\end{align*}

\smallskip
\noindent{\it (b) The critical or supercritical initial mass $m_0 \ge \frac{r}{\mu}$.}

By Lemma \ref{le2.1}, \eqref{4.34} becomes
\begin{align*}
    \frac{d}{dt} \int_{\Omega} u^p
    \le -\frac{3(p-1)}{2p} \int_{\Omega} |\nabla u^{\frac{p}{2}}|^2
    + \left( \frac{p-1}{p} + \frac{2L_1}{e} c_2 \right) \int_{\Omega} u^p
    + c_1 \int_{\Omega} u^{p+\gamma} + c_{28}.
\end{align*}
Young's inequality provides
\begin{align*}
    \int_{\Omega} u^p \le \int_{\Omega} u^{p+\gamma} + c(p).
\end{align*}
It follows that
\begin{align*}
    \frac{d}{dt} \int_{\Omega} u^p
    \le -\frac{3(p-1)}{2p} \int_{\Omega} |\nabla u^{\frac{p}{2}}|^2
    + c_{34} \int_{\Omega} u^{p+\gamma} + c(p,N),
\end{align*}
where $c_{34}:=\frac{p-1}{p}+\frac{2L_1}{e}c_2+c_1>0$.
Take $\phi = \gamma < \frac{2}{N}$ in \eqref{4.35.1} and combine with Young's inequality to get
\begin{align*}
    \int_{\Omega} u^{p+\gamma}
    \le \frac{p-1}{2c_{34} p} \int_{\Omega} |\nabla u^{\frac{p}{2}}|^2 + c(p,N),
\end{align*}
which implies
\begin{align*}
    \frac{d}{dt} \int_{\Omega} u^p
    \le -\frac{p-1}{p} \int_{\Omega} |\nabla u^{\frac{p}{2}}|^2 + c(p,N).
\end{align*}
Applying \eqref{4.12} together with Gronwall's inequality gives
\[
\sup_{0<t<T_{\max}}\int_\Omega u^p(x,t)\,dx
\leq c(p,N).
\]
$\hfill\Box$

\section{Global boundedness}
In this section, we prove Theorem \ref{th3.4}.
We first derive a positive lower bound for \(v\) under a uniform bound for \(u\)
and then obtain the \(L^\infty(\Omega)\)-estimate by the Moser iteration method.
Throughout this section, we denote by $c_i$ positive constants, which are renumbered anew starting from $c_1$.

\begin{lemma}\label{lem:v-lower}
Let \(\beta\geq1\) and \(0<\gamma\leq1\). Suppose that
\[
K:=
\max\left\{
1,\|u_0\|_{L^\infty(\Omega)},
\sup_{0<t<T_{\max}}
\|u(\cdot,t)\|_{L^\infty(\Omega)}
\right\}<\infty.
\]
Then there exists \(\delta_v>0\) such that
\[
v(x,t)\geq \delta_v
\]
for all \(x\in\Omega\) and \(t\in(0,T_{\max})\).
\end{lemma}
{\bf Proof.}
Denote the total mass by $m(t):=\int_\Omega u(x,t)\,dx$.
For \(\beta=1\), Lemma \ref{le2.1} directly provides
\[
m(t)\geq\min\left\{\int_\Omega u_0,\frac{r}{\mu}\right\}=:m_*>0.
\]
For \(\beta>1\), the upper bound for \(u\) yields
\[
\int_\Omega u^\beta\leq K^{\beta-1}m(t).
\]
Integrating the first equation in \eqref{1.3} over \(\Omega\) produces
\[
m'(t)=\left(r-\mu\int_\Omega u^\beta\right)\int_\Omega u^\alpha.
\]
The condition $m(t)\leq\frac{r}{\mu K^{\beta-1}}$
ensures that \(m'(t)\geq0\). An ODE comparison argument then gives
\[
m(t)\geq\min\left\{\int_\Omega u_0,\frac{r}{\mu K^{\beta-1}}\right\}=:m_*>0.
\]
On the other hand, proceeding as in the proof of Lemma 2.1 in \cite{ref49}
with \(u\) therein replaced by \(u^\gamma\), we infer that
\[
v(x,t)\geq\delta_0\int_\Omega u^\gamma(x,t)\,dx,
\]
where \(\delta_0>0\) depends only on \(\Omega\). Since
\(0\leq u\leq K\) and \(0<\gamma\leq1\), we have
\[
u^\gamma\geq K^{\gamma-1}u.
\]
Consequently, we deduce that
\[
v(x,t)
\geq
\delta_0K^{\gamma-1}m(t)
\geq
\delta_0K^{\gamma-1}m_*
=:\delta_v>0.
\]
This completes the proof.
\(\hfill\Box\)

We now prove Theorem \ref{th3.4}.\\
{\bf Proof of Theorem \ref{th3.4}.}
Fix \(p_*>\max\{N,2\}\). By Theorem \ref{th3.2}, we have
$$\sup_{0<t<T_{\max}}\|u(\cdot,t)\|_{L^{p_*}(\Omega)}<\infty.$$
Since \(0<\gamma\leq1\) under \(\mathcal{H}\), it follows that
$$
\sup_{0<t<T_{\max}}
\|u^\gamma(\cdot,t)\|_{L^{\frac{p_*}{\gamma}}(\Omega)}
=\sup_{0<t<T_{\max}}
\|u(\cdot,t)\|_{L^{p_*}(\Omega)}^\gamma<\infty.
$$
Due to \(\frac{p_*}{\gamma}>N\), elliptic regularity gives
$$
\sup_{0<t<T_{\max}}\|\nabla v(\cdot,t)\|_{L^\infty(\Omega)}\leq c_1<\infty,
$$
where \(c_1>0\) is independent of \(p\).

\par\medskip
\noindent{\bf I. The superlinear damping $\beta>1$.}

\smallskip
\noindent{\it I.1. The higher-dimensional setting $N\geq3$.}

Combining \eqref{4.1} and \eqref{4.2}, we have
\begin{align}\label{5.1}
    \frac{d}{dt} \int_{\Omega} u^p
    &\le -\frac{3(p-1)}{p} \int_{\Omega} |\nabla u^{\frac{p}{2}}|^2
    + \frac{p-1}{2p} \int_{\Omega} u^p
    + c_2 \frac{p-1}{p} p^{\frac{1}{1-k}} \int_{\Omega} u^{p+1} \nonumber\\
    &\quad + \chi k(p-1) \int_{\Omega} u^p \frac{|\nabla v|^2}{v^{k+1}}
    + rp \int_{\Omega} u^{p+\alpha-1}
    - \mu p \int_{\Omega} u^{p+\alpha-1} \int_{\Omega} u^\beta,
\end{align}
where $c_2 = 2^{\frac{k}{1-k}} \chi^{\frac{1}{1-k}}$.
Applying Young's inequality and Lemma \ref{le2.2} with $\lambda_4 = \frac{1}{4\chi p}$ yields
\begin{align}\label{5.2}
    k\chi(p-1)\int_{\Omega}u^p\frac{|\nabla v|^2}{v^{1+k}}
    &= k\chi(p-1)\int_{\Omega} \left(u^p\frac{|\nabla v|^2}{v^2}\lambda_4\right)^{\frac{1+k}{2}}
    \left(u^p|\nabla v|^2\lambda_4^{-\frac{1+k}{1-k}}\right)^{\frac{1-k}{2}} \nonumber\\
    &\le \frac{\chi k(1+k)(p-1)}{2}\lambda_4\int_{\Omega}u^p\frac{|\nabla v|^2}{v^2}
    + \frac{\chi k(1-k)(p-1)}{2}\lambda_4^{-\frac{1+k}{1-k}}\int_{\Omega}u^p|\nabla v|^2 \nonumber\\
    &\le \chi\lambda_4(p-1)\int_{\Omega}u^p\frac{|\nabla v|^2}{v^2}
    + \chi(p-1)\lambda_4^{-\frac{1+k}{1-k}}\int_{\Omega}u^p|\nabla v|^2 \nonumber\\
    &\le 4\chi\lambda_4(p-1)\int_{\Omega}|\nabla u^{\frac{p}{2}}|^2
    + \chi(p-1)\left(2\lambda_4 + \lambda_4^{-\frac{1+k}{1-k}}c_1^2\right)\int_{\Omega}u^p \nonumber\\
    &\le \frac{p-1}{p}\int_{\Omega}|\nabla u^{\frac{p}{2}}|^2
    + \left(\frac{p-1}{2p} + c_3\frac{p-1}{p}p^{\frac{2}{1-k}}\right)\int_{\Omega}u^p,
\end{align}
where $c_3 = 4^{\frac{1+k}{1-k}}\chi^{\frac{2}{1-k}}c_1^2$.
Substituting \eqref{5.2} into \eqref{5.1}, we find
\begin{align}\label{5.2.1}
    &\quad\frac{d}{dt} \int_{\Omega} u^p
    + \frac{2(p-1)}{p} \int_{\Omega} |\nabla u^{\frac{p}{2}}|^2
    + \mu p \int_{\Omega} u^{p+\alpha-1} \int_{\Omega} u^\beta \nonumber\\
    &\le \left(p + c_3 p^{\frac{2}{1-k}}\right) \int_{\Omega} u^p
    + c_2 p^{\frac{1}{1-k}} \int_{\Omega} u^{p+1}
    + rp \int_{\Omega} u^{p+\alpha-1},
\end{align}
which implies by Young's inequality that
\begin{align*}
    \frac{d}{dt} \int_{\Omega} u^p
    + \frac{2(p-1)}{p} \int_{\Omega} |\nabla u^{\frac{p}{2}}|^2
    + \mu p \int_{\Omega} u^{p+\alpha-1} \int_{\Omega} u^\beta
    &\le \left[(r+1)p + c_3 p^{\frac{2}{1-k}} + c_2 p^{\frac{1}{1-k}}\right] \int_{\Omega} u^{p+\xi-1} \\
    &\quad + \left[(r+1)p + c_3 p^{\frac{2}{1-k}} + c_2 p^{\frac{1}{1-k}}\right] c_4 \\
    &\le c_5 p^{\frac{2}{1-k}} \int_{\Omega} u^{p+\xi-1}
    + c_5 p^{\frac{2}{1-k}} c_4,
\end{align*}
where $\xi = \max\{\alpha, 2\}$, $c_4>0$ is independent of $p$ and $c_5 = 3\max\{r+1, c_2, c_3,1\}$.
Define $p_j = 2^j + \beta + \alpha - 1$ for \(j=0,1,2,\ldots\). Then we have
\begin{align}\label{5.3}
    \frac{d}{dt} \int_{\Omega} u^{p_j}
    + \frac{2(p_j-1)}{p_j} \int_{\Omega} |\nabla u^{\frac{p_j}{2}}|^2
    + \mu p_j \int_{\Omega} u^{p_j+\alpha-1} \int_{\Omega} u^\beta
    \le c_5 p_j^{\frac{2}{1-k}} \int_{\Omega} u^{p_j+\xi-1}
    + c_5 p_j^{\frac{2}{1-k}} c_4.
\end{align}
Taking $\omega = u^{\frac{p_j}{2}}$, $s = \frac{2(p_j+\xi-1)}{p_j}$, $l = \frac{2p_{j-1}}{p_j}$,
$C_{N_1} = \frac{1}{4c_5 p_j^{\frac{2}{1-k}}}$ and $C_{N_2} = \frac{1}{4c_5 p_j^{\frac{2}{1-k}}}$ in Lemma \ref{le2.5},
we get
\begin{align}\label{5.4}
    c_5 p_j^{\frac{2}{1-k}} \int_{\Omega} u^{p_j+\xi-1}
    \le C(N) p_j^{\frac{2}{1-k}} C_{N_1}^{-\frac{1}{\frac{2}{\eta s}-1}} \left( \int_{\Omega} u^{p_{j-1}} dx \right)^{\hat{\theta}}
    + \frac{1}{4} \int_{\Omega} |\nabla u^{\frac{p_j}{2}}|^2
    + \frac{1}{4} \int_{\Omega} u^{p_j}.
\end{align}
Here $\frac{2}{\eta s} = \frac{p_j - \frac{2p_{j-1}}{q}}{p_j+\xi-1 - p_{j-1}}$
and $\hat{\theta} = \frac{p_j - \frac{2(p_j+\xi-1)}{q}}{(1-\frac{2}{q})p_{j-1} - \xi + 1}$.
Substituting \eqref{5.4} into \eqref{5.3} and noting that $\frac{2(p_j-1)}{p_j} \ge 1$, we infer
\begin{align}\label{5.5}
    &\quad\frac{d}{dt} \int_{\Omega} u^{p_j}
    + \frac{3}{4} \int_{\Omega} |\nabla u^{\frac{p_j}{2}}|^2
    + \mu p_j \int_{\Omega} u^{p_j+\alpha-1} \int_{\Omega} u^\beta \nonumber\\
    &\le C(N) p_j^{\frac{2}{1-k}} C_{N_1}^{-\frac{1}{\frac{2}{\eta s}-1}} \left( \int_{\Omega} u^{p_{j-1}}
    \right)^{\hat{\theta}} + \frac{1}{2} \int_{\Omega} u^{p_j} + c_5 p_j^{\frac{2}{1-k}} c_4 \nonumber\\
    &\leq c_6 p_j^\delta \left( \int_{\Omega} u^{p_{j-1}} \right)^{\hat{\theta}}
    + \frac{1}{2} \int_{\Omega} u^{p_j}
    + c_5 p_j^{\frac{2}{1-k}} c_4,
\end{align}
where \(R_j:=\frac{2}{\eta s}\), $c_6:=C(N)\sup_{j\geq1}(4c_5)^{\frac{1}{R_j-1}}$ and $\delta:=\sup_{j\geq1}\frac{2R_j}{(R_j-1)(1-k)}$.
A direct calculation gives $R_j\to\frac{N+2}{N}>1$ and
hence \(c_6\) and \(\delta\) are finite constants independent of \(j\).
On the other hand, applying Lemma \ref{le2.5} with $\omega = u^{\frac{p_j}{2}}$, $s=2$,
$l = \frac{2p_{j-1}}{p_j}$, $C_{N_1} = \frac{1}{4}$, and $C_{N_2} = \frac{1}{2}$, we obtain
\begin{align}\label{5.6}
    \frac{1}{2} \int_{\Omega} u^{p_j}
    \le C(N) \left( \int_{\Omega} u^{p_{j-1}} \right)^{\frac{p_j}{p_{j-1}}}
    + \frac{1}{4} \int_{\Omega} |\nabla u^{\frac{p_j}{2}}|^2.
\end{align}
By Lemma \ref{le2.6}, we derive
\begin{align*}
    C(N) \left( \int_{\Omega} u^{p_{j-1}} \, dx \right)^{\frac{p_j}{p_{j-1}}}
    \le C(N) \left( \int_{\Omega} u^\beta \right)^{\frac{p_j(1-\sigma)}{\beta}}
    \left( \int_{\Omega} u^{p_j+\alpha-1} \right)^{\frac{\sigma p_j}{p_j+\alpha-1}},
\end{align*}
where $\sigma$ is determined by $\frac{1-\sigma}{\beta} + \frac{\sigma}{p_j+\alpha-1} = \frac{1}{p_{j-1}}$.
The identity $p_{j-1} = \frac{p_j+\alpha-1+\beta}{2}$ implies
$\frac{p_j(1-\sigma)}{\beta} = \frac{\sigma p_j}{p_j+\alpha-1}$.
It then follows from Young's inequality that
\begin{align}\label{5.7}
    C(N) \left( \int_{\Omega} u^{p_{j-1}} \, dx \right)^{\frac{p_j}{p_{j-1}}}
    &\le C(N) \left( \int_{\Omega} u^\beta \int_{\Omega} u^{p_j+\alpha-1} \right)^{\frac{\sigma p_j}{p_j+\alpha-1}} \nonumber\\
    &\le \frac{\mu p_j}{4} \int_{\Omega} u^\beta \int_{\Omega} u^{p_j+\alpha-1} + c_7.
\end{align}
Plugging \eqref{5.7} into \eqref{5.6}, we obtain
\begin{align}\label{5.8}
    \frac{1}{2} \int_{\Omega} u^{p_j}
    \le \frac{\mu p_j}{4} \int_{\Omega} u^{p_j+\alpha-1} \int_{\Omega} u^\beta
    + \frac{1}{4} \int_{\Omega} |\nabla u^{\frac{p_j}{2}}|^2 + c_7.
\end{align}
The estimate $\hat{\theta}\leq2$ is verified directly under $\mathrm{(H_1)}$.
Combining \eqref{5.8} with \eqref{5.5}, we then have
\begin{align*}
    \frac{d}{dt} \int_{\Omega} u^{p_j} + \int_{\Omega} u^{p_j}
    &\le c_6 p_j^\delta \left( \int_{\Omega} u^{p_{j-1}} \right)^{\hat{\theta}}
    + c_5 p_j^{\frac{2}{1-k}} c_4 + c_8 \\
    &\le c_6 p_j^\delta \left( \int_{\Omega} u^{p_{j-1}} \right)^{\hat{\theta}}
    + c_9 p_j^{\frac{2}{1-k}} \\
    &\le \max\{c_6, c_9\} \, p_j^b \left[ \left( \int_{\Omega} u^{p_{j-1}} \right)^{\hat{\theta}} + 1 \right] \\
    &\le 2\max\{c_6, c_9\} \, p_j^b \max\left\{ \left( \int_{\Omega} u^{p_{j-1}} \right)^2, \, 1 \right\},
\end{align*}
where $b = \max\left\{ \delta, \frac{2}{1-k} \right\}$.
Applying Lemma \ref{le2.9} with $y_j(t) = \int_{\Omega} u^{p_j}$, $\tilde{m} = 2\max\{c_6, c_9\} (\alpha+\beta)^b$
and $h^{bj} = 2^{bj}$, we arrive at
\begin{align}\label{5.9}
\int_{\Omega} u^{p_j}
\le (2\tilde{m})^{2^j-1} 2^{b(2^{j+1}-2-j)}
\max\left\{ \sup_{0<t<T_{\max}} \left( \int_{\Omega} u^{p_0} \right)^{2^j}, \tilde{M}^{2^j} \right\}.
\end{align}
Here we set $M_1:=\max\left\{1,\|u_0\|_{L^{\alpha+\beta}(\Omega)},\|u_0\|_{L^\infty(\Omega)}\right\}$
and $\tilde M:=M_1^{p_0}$.
The fact that \(p_j\leq p_0 2^j\) implies
\[
y_j(0)=\int_\Omega u_0^{p_j}
\leq M_1^{p_j}
\leq\tilde M^{2^j}.
\]
Taking the $\frac{1}{p_j}$-th power in \eqref{5.9} and letting $j\to\infty$, we obtain
\begin{align*}
\|u\|_{L^\infty(\Omega)}\le 2^{2b+1}\tilde{m}
\max\left\{\sup_{0<t<T_{\max}}\int_\Omega u^{p_0},\tilde M\right\}.
\end{align*}
Theorem \ref{th3.2} gives
\begin{align*}
\int_{\Omega} u^{p_0} = \int_{\Omega} u^{\beta+\alpha} \le c(p_0).
\end{align*}

\smallskip
\noindent{\it I.2. The lower-dimensional setting $N\in\{1,2\}$.}

Set \(p_j=2^j+\alpha+\beta+1\).
Applying Lemma~\ref{le2.5} as in \eqref{5.4} gives
$R_j=\frac{p_j-\left(1-\frac2N\right)p_{j-1}}{p_j+\xi-1-p_{j-1}}$ and
$\hat{\theta}=\frac{p_j-\left(1-\frac2N\right)(p_j+\xi-1)}{\frac2N p_{j-1}-\xi+1}$.
A direct calculation under \(\mathrm{(H_1)}\) shows that $R_j\to\frac{N+2}{N}>1$
and $\hat{\theta}\leq2$.
Repeating the argument leading to \eqref{5.6} with $l=\frac{2(p_{j-1}-1)}{p_j}$
replaces \(p_{j-1}\) by \(p_{j-1}-1\) on its right-hand side.
Since $p_{j-1}-1=\frac{p_j+\alpha+\beta-1}{2}$,
the same application of Lemma \ref{le2.6} yields \eqref{5.8}.
Theorem \ref{th3.2} ensures the initial estimate at \(p_0=\alpha+\beta+2\), and the iteration in \eqref{5.9} completes this case.

\par\medskip
\noindent{\bf II. The linear damping $\beta=1$.}

\smallskip
\noindent{\bf II.1. The subcritical initial mass $m_0 < \frac{r}{\mu}$.}\par

\smallskip
\noindent{\it (a) The one-dimensional setting $N=1$.}
Here $\gamma=1$.
Define $p_j = 2^j + \xi + 1$, where $\xi = \max\{\alpha, 2\}$. Arguing as in the derivation of \eqref{5.3}, we have
\begin{align*}
\frac{d}{dt} \int_{\Omega} u^{p_j}
+ \frac{2(p_j-1)}{p_j} \int_{\Omega} |\nabla u^{\frac{p_j}{2}}|^2
+ \mu p_j \int_{\Omega} u^{p_j+\alpha-1} \int_{\Omega} u
\le c_5 p_j^{\frac{2}{1-k}} \int_{\Omega} u^{p_j+\xi-1}
+ c_5 p_j^{\frac{2}{1-k}} c_4.
\end{align*}
As in the case $\beta>1$, an application of Lemma \ref{le2.5} yields
\begin{align}\label{5.10}
\frac{d}{dt} \int_{\Omega} u^{p_j}
+ \frac{3}{4} \int_{\Omega} |\nabla u^{\frac{p_j}{2}}|^2
+ \mu p_j \int_{\Omega} u^{p_j+\alpha-1} \int_{\Omega} u
\le c_{10} p_j^\delta \left( \int_{\Omega} u^{p_{j-1}} \right)^{\hat{\theta}}
+ \frac{1}{2} \int_{\Omega} u^{p_j}
+ c_5 p_j^{\frac{2}{1-k}} c_4,
\end{align}
where $\tilde R_j:=\frac{2}{\eta s}=\frac{p_j+p_{j-1}}{p_j+\xi-1-p_{j-1}}$
and $\hat{\theta}=\frac{2p_j+\xi-1}{2p_{j-1}-\xi+1}$.
Set $c_{10}:=C(N)\sup_{j\geq1}(4c_5)^{\frac{1}{\tilde R_j-1}}$
and $\delta:=\sup_{j\geq1}\frac{2\tilde R_j}{(\tilde R_j-1)(1-k)}$.
Since \(\tilde R_j\to3>1\), the constants
\(c_{10}\) and \(\delta\) are finite and independent of \(j\).
Using \(\frac{p_j}{p_{j-1}}\leq2\) together with
\(z^\varrho\leq z^2+1\) for \(z\geq0\) and \(0<\varrho\leq2\),
the same argument as in \eqref{5.6} gives
\begin{align}\label{5.11}
\frac12\int_\Omega u^{p_j}
\leq c_{10}p_j^\delta
\left(\int_\Omega u^{p_{j-1}}\right)^2
+\frac14\int_\Omega
|\nabla u^{\frac{p_j}{2}}|^2+c_{11},
\end{align}
where $c_{11}>0$ is independent of $j$.
Since $1\leq\alpha<3$, a direct calculation gives $2-\hat{\theta}=\frac{5-\xi}{2p_{j-1}-\xi+1}>0$ and thus \(\hat{\theta}<2\).
Combining \eqref{5.10} with \eqref{5.11}
yields the same differential inequality as in the case
$\beta>1$. This iteration then provides the uniform $L^\infty$-estimate.

\smallskip
\noindent{\it (b) The higher-dimensional setting $N\geq2$.}
Here $0<\gamma<\frac{2}{N}$.
Combining \eqref{4.32} and \eqref{5.2}, for every \(p\geq2\), we obtain
\begin{align}\label{6.1}
\frac{d}{dt} \int_{\Omega} u^{p}
&\le -\frac{2(p-1)}{p} \int_{\Omega} |\nabla u^{\frac{p}{2}}|^2
+ \left( \frac{p-1}{p} + c_3 \frac{p-1}{p} p^{\frac{2}{1-k}} \right) \int_{\Omega} u^{p}\nonumber\\
&\quad+ c_2 p^{\frac{1}{1-k}} \int_{\Omega} u^{p+\gamma}
+ (r-\mu m(t)) p \int_{\Omega} u^{p+\alpha-1},
\end{align}
where $c_2,c_3>0$ are defined above.
By Young's inequality, we infer that
\begin{align*}
\frac{d}{dt} \int_{\Omega} u^{p}
\le -\frac{2(p-1)}{p} \int_{\Omega} |\nabla u^{\frac{p}{2}}|^2
+ c_{12} p^{\frac{2}{1-k}} \int_{\Omega} u^{p+\varphi}
+ c_{12} p^{\frac{2}{1-k}} c_{13},
\end{align*}
where $c_{12} = 1 + c_3 + c_2 + r - \mu m_0$, $c_{13}>0$ is independent of $p$ and $\varphi = \max\{\gamma, \alpha-1\}$.
Let \(p_j=2^j+\varphi+1\). For \(N=2\), Lemma \ref{le2.5} gives
$R_j\to2$ and $\hat{\theta}-2=\frac{\varphi-1}{p_{j-1}-\varphi}<0$.
For \(N\geq3\), one has
$\hat{\theta}-2=\frac{\varphi-\frac2N}{\frac2N p_{j-1}-\varphi}<0$.
Thus the preceding Moser iteration applies,
with the initial estimate provided by Theorem \ref{th3.2}.

\smallskip
\noindent{\bf II.2. The critical or supercritical initial mass $m_0 \geq \frac{r}{\mu}$.}\par

\smallskip
\noindent{\it (a) The one-dimensional setting $N=1$.}

Here $\gamma=1$.
It follows from \eqref{5.2.1} that
\begin{align*}
\frac{d}{dt} \int_{\Omega} u^p
+ \frac{2(p-1)}{p} \int_{\Omega} |\nabla u^{\frac{p}{2}}|^2
+ (\mu m(t)-r)p \int_{\Omega} u^{p+\alpha-1}
\le \left( p + c_3 p^{\frac{2}{1-k}} \right) \int_{\Omega} u^p
+ c_2 p^{\frac{1}{1-k}} \int_{\Omega} u^{p+1}.
\end{align*}
Using $\frac{r}{\mu} \le m(t) \le m_0$ and Young's inequality, we obtain
\begin{align*}
\frac{d}{dt} \int_{\Omega} u^p
+ \frac{2(p-1)}{p} \int_{\Omega} |\nabla u^{\frac{p}{2}}|^2
\le c_{14} p^{\frac{2}{1-k}} \int_{\Omega} u^{p+1}
+ c_{15} p^{\frac{2}{1-k}},
\end{align*}
where $c_{14} = c_3 + c_2 + 1$ and $c_{15}>0$ is independent of $p$.
For \(p_j=2^j+3\), Lemma \ref{le2.5} gives $R_j\to3$
and $\hat{\theta}=\frac{2p_j+1}{2p_{j-1}-1}<2$.
The uniform \(L^\infty\)-bound is established by the same iteration.

\smallskip
\noindent{\it (b) The higher-dimensional setting $N\geq2$.}

Here $0<\gamma<\frac{2}{N}$ and $r-\mu m(t)\leq0$.
Taking $p=p_j$ in \eqref{6.1} and applying Young's inequality, we derive
\begin{align*}
\frac{d}{dt}\int_{\Omega}u^{p_j}
&\leq-\frac{2(p_j-1)}{p_j}\int_{\Omega}|\nabla u^{\frac{p_j}{2}}|^2
+c_{16}p_j^{\frac{2}{1-k}}\int_{\Omega}u^{p_j+\gamma}+c_{16}p_j^{\frac{2}{1-k}},
\end{align*}
where \(c_{16}>0\) is independent of \(j\).
Let \(p_j=2^j+\gamma+1\). For \(N=2\), Lemma \ref{le2.5} gives
$R_j=\frac{p_j}{p_j+\gamma-p_{j-1}}\to2$
and $\hat{\theta}-2=\frac{\gamma-1}{p_{j-1}-\gamma}<0$.
For \(N\geq3\), we find that
$\hat{\theta}-2=\frac{\gamma-\frac2N}{\frac2N p_{j-1}-\gamma}<0$
as \(\gamma<\frac2N\).
The uniform \(L^\infty\)-estimate is provided by the preceding iteration.

If \(T_{\max}<\infty\), the uniform \(L^\infty\)-estimate for \(u\)
and Lemma \ref{lem:v-lower} contradict the extensibility criterion
in Proposition \ref{pro4.1}.
Thus \(T_{\max}=\infty\), and the solution of \eqref{1.3} is global and uniformly bounded.
$\hfill\Box$

\section*{Conflicts of Interest}
The authors have no conflict of interest to declare.



\end{document}